\documentclass[12pt]{amsart}

\usepackage{amsmath, amssymb, amscd, newlfont}

\newcommand{\bbA}{{\mathbb{A}}}

\newcommand{\bbC}{{\mathbb{C}}}

\newcommand{\bbZ}{{\mathbb{Z}}}

\newcommand{\supp}{{\mathrm{supp}}}

\newcommand{\trace}{{\mathbf{tr}}}

\usepackage[mathscr]{eucal}

\numberwithin{equation}{section}

\newtheorem{Prop}[equation]{Proposition}
\newtheorem{Lem}[equation]{Lemma}
\newtheorem{Def}[equation]{Definition}
\newtheorem{Thm}[equation] {Theorem}
\newtheorem{Cor}[equation]{Corollary}

\title
    [On Fourier Coefficients of Automorphic Forms]
    {Fourier Coefficients of Automorphic Forms and  
    Integrable Discrete Series}

\author{ Goran Mui\' c}

\address{ Department of Mathematics,
University of Zagreb,
Bijeni\v cka 30, 10000 Zagreb,
Croatia}
\email{gmuic@math.hr}

\thanks{The  author acknowledges Croatian Science Foundation grant no. 9364.}

\begin{document}
\maketitle

\begin{abstract}
Let $G$ be the group of $\mathbb R$--points of a semisimple algebraic group $\mathcal G$ defined over $\mathbb Q$. 
Assume that $G$ is  connected and noncompact. We study Fourier coefficients of Poincar\' e series attached to matrix
coefficients 
of integrable discrete series. We use these results to construct explicit automorphic cuspidal realizations,
which have appropriate Fourier coefficients $\neq 0$, of 
integrable discrete series  in families of congruence subgroups.  In the case of $G=Sp_{2n}(\mathbb R)$, we relate
our work to that of Li \cite{Li}. For $\mathcal G$ quasi--split over $\mathbb Q$, we relate our work 
to the result about Poincar\' e series due to Khare, Larsen, and Savin \cite{kls}.
\end{abstract} 
\section{Introduction}

Let $G$ be the group of $\mathbb R$--points of a semisimple algebraic group $\mathcal G$ defined over $\mathbb Q$. 
Assume that $G$ is  connected and noncompact.
Let $K$ be its maximal compact subgroup, $\mathfrak g$ be the real Lie algebra of $G$, and $\mathcal Z(\mathfrak g_{\mathbb C})$ the center of the 
universal enveloping algebra of the complexification of $\mathfrak g$.  In this paper we assume that $rank(K)=rank(G)$ so that $G$ 
admits discrete series. Let $\Gamma$ be a discrete subgroup of finite
covolume in $G$. The problem of (mostly asymptotic)
realization of discrete series as subrepresentations in the discrete part of the spectrum of
$L^2(\Gamma \backslash G)$ has been studied extensively using various methods such as cohomology of discrete subgroups and
adelic Arthur trace formula \cite{savin},  \cite{clozel}, \cite{rospeh}, \cite{degeorgy}, \cite{finis}. 
But Fourier coefficients of such realizations are not well understood. On the other hand, assuming that $\mathcal G$ is a 
quasi-split almost simple algebraic group over $\mathbb Q$, for any appropriate generic integrable discrete series $\pi$ of $G$, 
Khare, Larsen, and Savin (\cite{kls} Theorem 4.5) construct globally generic  
automorphic cuspidal representation $W$ of  $\mathcal G(\mathbb A)$, where $\mathbb A$ is the ring of adeles of $\mathbb Q$, 
such that its Archimedean component is $\pi$ i.e.,  $W_\infty=\pi$. They use classical theory of Poincar\' e series attached to matrix coefficients of 
$\pi$ which are $\mathcal Z(\mathfrak g_{\mathbb C})$--finite  and $K$--finite on the right extended to  adelic settings 
(\cite{bb}, \cite{Borel-SL(2)-book}, \cite{Borel1966}). A detailed  treatment of such series and criteria for being $\neq 0$ 
in the adelic settings can be  found  in \cite{MuicMathAnn}.  In  (\cite{kls} Theorem 4.5), Khare, Larsen, and Savin in fact prove an 
analogue of  a well--known result of Vign\' eras, Henniart and  Shahidi \cite{sh} for generic supercuspidal representations of semisimple $p$--adic  
groups. (\cite{kls} Theorem 4.5) is used to study problems in inverse Galois theory. 
On the other hand,  (possibly degenerate) Fourier coefficients of automorphic forms are important for the theory of automorphic
$L$--functions \cite{sh}, \cite{grs}, \cite{soudry}. So, it is important to study (possibly degenerate) Fourier coefficients of 
 Poincar\' e series attached to matrix coefficients of  integrable discrete series which are $K$--finite.  This is the goal of the present paper. 
The techniques used in this paper are refinements of those of Khare, Larsen, and Savin used in the proof of  (\cite{kls} Theorem 4.5), and those of 
(for adelic or for real groups) used in (\cite{MuicJNT}, \cite{MuicLie}, \cite{MuicComp}, \cite{MuicIJNT}.  We do not improve only on  results in generic 
case (see Theorems \ref{cps-13} and  \ref{cps-16})  but we also  construct very
explicit degenerate cuspidal automorphic models of integrable discrete series \cite{Li} for $Sp_{2n}(\mathbb R)$
(see Theorem \ref{cps-18}).

The paper actually has two parts: preliminary local Archimedean part
(Sections \ref{cbr}, \ref{gwm}, \ref{integrable}), and main cuspidal automorphic part (Sections \ref{paf}, \ref{cps}, \ref{app}).
Integrable discrete series for $G$ are analogues of supercupidal representations for  semi--simple
$p$--adic groups. In Section  \ref{gwm} we refine and generalize  (\cite{kls}, Proposition 4.2) discussing the  analogue of the
following simple result for $p$--adic groups which we explain in detail. 
Assume for the moment that  $G$ is a semisimple $p$--adic group, $U$ closed unimodular
subgroup of $G$, and
$\chi: U\longrightarrow \mathbb C^\times$ a unitary character. Let $(\pi, \mathcal H)$ be
an irreducible supercuspidal representation acting on the Hilbert space $\mathcal H$. Let us write $\mathcal H^\infty$ for the space of smooth
vectors (i.e., the subspace of all vectors in $\mathcal H$ which have open stabilizers). The space $\mathcal H^\infty$ is
usual smooth irreducible algebraic representation of $G$.  We say
that $\pi$ is $(\chi, U)$--generic if there exists a non--zero  (algebraic) functional
  $\lambda: \ \cal H^\infty\longrightarrow \mathbb C$ which satisfies
$\lambda(\pi(u)h)=\chi(u)\lambda(h)$,  $u\in U$, $h\in \mathcal H^\infty$. Let $\varphi$ be a non--zero matrix coefficient
of $(\pi,  \mathcal H^\infty)$. Then, $\varphi\in C_c^\infty(G)$. Next, a simple argument (see Lemma \ref{gwm-1}),
first observed by Mili\v ci\' c in his unpublished lecture notes about
$SL_2(\mathbb R)$, Schur orthogonality  relations show that for each $h\in \mathcal H^\infty$, we can select a matrix coefficient 
$\varphi$ such that  $\pi(\overline{\varphi})h=h$. Now, since $h=\pi(\overline{\varphi})h=\int_G \overline{\varphi(g)}\pi(g)h dg$ is essentially a 
finite sum, we have 
 $$
\lambda(h)=\lambda(\pi(\overline{\varphi})h)=  \int_G \overline{\varphi(g)} \lambda(\pi(g)h)dg=
\int_{U\backslash G} \left(\int_U \overline{\varphi(ug)}\chi(u)du\right) \lambda(\pi(g)h)dg.
    $$
So, selecting $h$ such that $\lambda(h)\neq 0$, we have that $\varphi$ has a non--zero Fourier coefficient
\begin{equation}\label{FC}
\mathcal F_{(\chi, U)}^{loc}(\varphi)(g)\overset{def}{=} \int_U \varphi(ug)\overline{\chi(u)}du, \ \  g \in G.
\end{equation}

Going back to our original meaning for $G$  and  assuming that $(\pi, \cal H)$ is an integrable discrete series acting on a Hilbert
space  $\mathcal H$, we try to adjust these consideration  where now  $\varphi$ is a $K$--finite matrix coefficient of $\pi$
and $\lambda$ is continuous in appropriate Fr\' echet topology on $\mathcal H^\infty$ (\cite{W3}; see  Section \ref{gwm}  for details).

Since $\varphi$ is no longer compactly supported, the main obstacle  to the extension of above result is how to
exchange  $\lambda$ and $\int_G$. This boils down to absolute convergence of the integral
$ \int_G \overline{\varphi(g)} \lambda(\pi(g)h)dg$ (see Lemma \ref{gwm-3}).
We attack this problem using results of  Casselman \cite{casselman}  and Wallach \cite{W2} on representations of moderate
growth, and  the estimates on the growth of $K$--finite matrix coefficients of discrete series due to Mili\v ci\' c \cite{milicic}.
For general $U$ and $\chi$, Theorem \ref{gwm-5} contains the result.

If we in addition assume that $U=N$ is the unipotent radical of a minimal parabolic subgroup  of $G$, and $\chi$ is a generic 
character of $N$ (i.e.,  the differential $d\chi$ is non--trivial
  on any simple root subgroup $\mathfrak n_\alpha\subset Lie(N)$), Wallach (\cite{W2}, Theorem 15.2.5) shows that the asymptotics of the function
$g\longmapsto \lambda(\pi(g)h)$ is similar to the asymptotics of a $K$--finite matrix coefficient. This implies that 
$g\longmapsto \lambda(\pi(g)h)$ is bounded on $G$ (see the proof of Corollary \ref{gwm-6}). 
Since $\varphi \in L^1(G)$ being a $K$--finite matrix coefficient of an integrable  discrete series, 
the absolute convergence $ \int_G \overline{\varphi(g)} \lambda(\pi(g)h)dg$  is obvious. So, in this case 
there exists a non--zero $K$--finite matrix
  coefficient $\varphi$ of $\pi$ such that $\mathcal F_{(\chi, N)}^{loc}(\varphi)\neq 0$ (see Corollary \ref{gwm-6}). 
In  (\cite{kls}, Proposition 4.2), Khare, Larsen, and Savin obtain slightly weaker result (for some  matrix coefficient of $\pi$ 
$K$--finite on the right only) for general discrete series relying again on results of Wallach but they use more advanced
results of the same chapter in \cite{W3}.

Going back to general $U$ and $\chi$, the existence of non--zero $K$--finite matrix
coefficients $\varphi$ of $\pi$ such that $\mathcal F_{(\chi, U)}^{loc}(\varphi)\neq 0$ is used to show that $\pi$ is
infinitesimally
  equivalent to a closed irreducible subrepresentation of certain Banach representation $I^1(G, U, \chi)$
  (see Section \ref{cbr} for definition, Corollary \ref{gwm-7}).
  This observation is of a fundamental importance  in the second part of the paper. We do not study absolute
convergence of the 
 integral defining  $\mathcal F_{(\chi, N)}^{loc}(\varphi)$ but we use some general representation theory to show that it 
 converges absolutely almost everywhere on $G$ to a function in  $C^\infty(G)$ which is $Z(\mathfrak g_{\mathbb C})$--finite and
 $K$--finite on the 
right  (see Lemma \ref{cps-5}). When $U=N$,  the absolute convergence of the integral is a well--known
classical result of
Harish--Chandra.

In the appendix to the paper (see Section  \ref{integrable})
we describe large integrable representations in usual $L$--packets of discrete series. The results  of this section were
communicated to us by Mili\v ci\' c.  It shows that many discrete series are large and integrable.

In the second part of the paper (Sections \ref{paf}, \ref{cps}, \ref{app}), we consider automorphic forms and prove
the main results of the paper. Let $\Gamma\subset G$ be congruence subgroup with respect to the arithmetic structure given by the
fact that $\mathcal G$ defined over $\mathbb Q$ (see \cite{BJ}, or Section \ref{app}).
In Section \ref{paf}, we review
basic notions of automorphic forms, and introduce a classical construction of automorphic forms via Poincar\' e series 
$$
P(\varphi)(x)=P_\Gamma(\varphi)(x)=
\sum_{\gamma\in \Gamma}\varphi(\gamma x), \ \ x\in G,
$$ 
attached to  $K$--finite matrix coefficients $\varphi$ of  integrable 
discrete series $\pi$ of $G$ (\cite{bb}, \cite{Borel-SL(2)-book}, \cite{Borel1966}).
We recall the criteria for $P(\varphi)$ being non--zero (\cite{MuicMathAnn}, Theorem  4-1 and  4-9).

In Sections  \ref{cps} and \ref{app}, we 
assume that $\chi$ and $U$ are as above, with addition that $U$ is the group of $\mathbb R$--points of the
unipotent radical of a proper $\mathbb Q$--parabolic $\mathcal P\subset \mathcal G$. Then,  $U\cap \Gamma$ is cocompact in $U$.
We define the $(\chi, U)$--Fourier coefficient of a function 
$f\in C^\infty(\Gamma\backslash G)$ as usual 

$$
\mathcal F_{(\chi, U)}(f)(x)=\int_{U\cap \Gamma\backslash U} f(ux)\overline{\chi(u)}du,
$$
where we use a normalized measure on a compact topological space $U\cap \Gamma\backslash U$. If $W$ is a
$(\mathfrak g, K)$--submodule of the space of automorphic forms, then we say that $W$ is $(\chi, U)$--generic
if there exists $f\in W$ such that $\mathcal F_{(\chi, U)}(f)\neq 0$.

The main point of
Section  \ref{cps} is to compute the Fourier coefficient $\mathcal F_{(\chi, U)}(P(\varphi))$ of $P(\varphi)$, and use
the results to
study when $P(\varphi)\neq 0$ for some new cases not covered by the results of \cite{MuicMathAnn}, 
and for construction of $(\chi, U)$--generic automorphic realizations of $\pi$. The main result of Section \ref{cps} is
the following
proposition (see Proposition \ref{cps-10}):

\begin{Prop}\label{prop-1}
 Let $\pi$ be an integrable discrete series of $G$. 
  Assume that there exists  a $K$--finite  matrix coefficient
  $\varphi$ of $\pi$ such that the following holds:
  $$
  W^\Gamma_{(\chi, U)}\left(\mathcal F_{(\chi, U)}^{loc}(\varphi)\right)\neq 0.
  $$
Then, there is a realization of  $\pi$  as a $(\chi, U)$--generic cuspidal  automorphic  representation.
\end{Prop}

As we explained above $\psi\overset{def}{=} \mathcal F_{(\chi, U)}^{loc}(\varphi)\in I^1(G, U, \chi)$. Next, 
the function  $W_{(\chi, U)}(\psi)$ is an automorphic form for $\Gamma$ defined via the following series
(see Lemma \ref{cps-7}):
$$
W_{(\chi, U)}(\psi)(x)=W^\Gamma_{(\chi, U)}(\psi)(x)\overset{def}{=}\sum_{\gamma \in U\cap \Gamma\backslash \Gamma}
\psi(\gamma x), \ \ x\in G.
$$
This statement   is especially important for construction $(\chi, U)$--generic automorphic cuspidal realization of
$\pi$ in view of a criterion which gives a sufficient condition that the series of this form (and more general) are not
identically zero (\cite{MuicIJNT}, Lemma 2-1, recalled as  Lemma \ref{nv-thm} in Section \ref{app}, and its formulation
given by Lemma \ref{cps-12} for series $W_{(\chi, U)}(\psi)$). We prove the following result (see Theorem \ref{cps-13}):

\begin{Thm} We fix an embedding $\mathcal G\hookrightarrow SL_M$ over $\mathbb Q$, and define 
Hecke congruence subgroups $\Gamma_1(n)$,
  $n\ge 1,$ using that embedding (see (\ref{int-0c-1})). Assume that $U$ is  a  subgroup of all upper triangular unipotent matrices in
 $G$ considered as $G 
\subset SL_M(\mathbb R)$. Let 
  $\chi$ be a unitary character $U\longrightarrow  \mathbb C^{\times}$ trivial on $U\cap  \Gamma_1(l)$ for some $l\ge 1$.
Let $\pi$ be an integrable discrete series of $G$  such that
 there exists  a  $K$--finite  matrix coefficient $\varphi$ of $\pi$ such that $\mathcal F_{(\chi, U)}^{loc}(\varphi)\neq 0$.
Then, there exists $n_0\ge 1$ such that for $n\ge n_0$ we have
 a realization of  $\pi$  as a $(\chi, U)$-generic cuspidal automorphic 
representation for $\Gamma_1(ln)$.
\end{Thm}

The reader may observe that $U\cap  \Gamma_1(ln)$ is independent of $n\ge 1$. So, the same $\chi$ can be used for all $\Gamma_1(ln)$. 
Also, the reader may observe that the congruence 
subgroups $\Gamma_1(ln)$ ($n\ge 1$) are neither  linearly ordered by inclusion nor their intersection is trivial as it is usually required
(see for example \cite{savin} and \cite{degeorgy}).

As a next application,   we reprove the result of Khare, Larsen, and Savin (\cite{kls}, Theorem 4.5) 
that large integrable discrete series are local components of globally generic automorphic cuspidal forms (see Theorem  \ref{cps-16}):

\begin{Thm} Let $\mathbb A$ be the ring of adeles of $\mathbb Q$.   We assume that $\mathcal G$ is quasi--split
  over $\mathbb Q$. Let $\mathcal N$ be the unipotent radical of Borel subgroup defined over $\mathbb Q$.  We assume that
  $G$ poses representations in discrete series. Let $L$ be the $L$--packet of discrete series for  $G$ such that there is a
  large
  representation in that packet which is integrable (then all are integrable by Proposition \ref{integrable-4}).
  Let $\eta: \mathcal N(\mathbb Q) \backslash \mathcal N(\mathbb A)\longrightarrow
  \mathbb C^\times$ be a unitary generic character. By the change of splitting we can select an
  $(\eta_\infty, \mathcal N(\mathbb R))$--generic discrete series $\pi$ in the $L$--packet $L$. Then, there exists
  a cuspidal automorphic module $W$ for $\mathcal G(\mathbb A)$ which is $\eta$--generic and its Archimedean component is
  infinitesimally equivalent to $\pi$ (i.e., considered as a $(\mathfrak g, K)$--module only it is a copy of (infinitely many)
  $(\mathfrak g, K)$--modules infinitesimally equivalent to $\pi$).
    \end{Thm}

In the very last part of Section \ref{app} (see Theorem \ref{cps-18}),  we consider a sort of a  converse to approach described above. 
We use criteria for $P(\varphi)\neq 0$ given by \cite{MuicMathAnn}, 
where $\varphi$ is a $K$--finite matrix coefficient of an integrable discrete series $\pi$ of 
$G=Sp_{2n}(\mathbb R)$, and show that there exists  matrix coefficients $\varphi$  with non--zero Fourier coefficients 
$\mathcal F_{(\chi, U)}^{loc}(\varphi)$  for certain unitary characters $\chi$ where  $U$ is the unipotent radical of  a Siegel parabolic subgroup of 
$Sp_{2n}(\mathbb R)$. 
In this way, any integrable discrete series $\pi$ of $Sp_{2n}(\mathbb R)$ is infinitesimally
equivalent to a closed irreducible subrepresentation of a  Banach representation $I^1(G, U, \chi)$ for certain $\chi$ (see Section \ref{cbr}). 
This has relation to the work of Li \cite{Li}, and it is related to the models discussed by Gourevitch \cite{gourevitch}. 

\vskip .2in
In closing the introduction, we would like to say that purpose of this paper is to write down typical applications in sufficient generality and to 
convince the audience how the methods of our earlier papers 
such as \cite{MuicMathAnn}, \cite{MuicComp}, and \cite{MuicIJNT}, 
are also useful for proving existence or constructing explicit examples 
of automorphic cuspidal representations with non--zero Fourier coefficients.  We do not write down an
 exhaustive list of applications. For example, using methods of \cite{MuicComp}, 
in ''Whittaker model case'' we can do much more than it is stated in Theorem  \ref{cps-16}. 

\vskip .2in
We would like to thank  A. Moy, N. Grbac, G. Savin,  and J. Schwermer for some useful
discussions. Especially, I would like to thank D. Mili\v ci\' c for his help with description of large and integrable
representations. The final version of this paper is prepared while I was visiting the Erwin
Schroedinger Institute in Vienna. I would like to thank J. Schwermer and the Institute 
for their hospitality.

\section{On Certain Banach Representations}\label{cbr}

In this section we assume that  $G$ is a connected semisimple Lie group with finite center,
$U$ is a closed unimodular subgroup of $G$, and 
$\chi: U\rightarrow \mathbb C^{\times}$ unitary character.
We state some simple results about the following Banach representation that will be used in the remainder of the paper.

We consider the Banach representation
$I^1(G, U, \chi)$ (by right translations) 
on the space of classes  of all measurable functions $f:G\rightarrow \mathbb C$ which satisfies
the following two conditions:

\begin{itemize}
\item[(i)] $f(ux)=\chi(u)f(x)$ for all $u\in U$, and for almost all $x\in G $,
\item[(ii)] $||f||_{U, 1}\overset{def}{=}\int_{U\backslash G} \left|f(x)\right| dx <\infty$.
\end{itemize}

The following is a folklore lemma, which might be useful for
unexperienced reader. We include it with a complete proof.

\begin{Lem}\label{cbr-1}
  The function $\varphi\in C_c^\infty(G)$ acts on $I^1(G, U, \chi)$ as follows: $\varphi.f(x)=\int_G \varphi(y) f(xy)dy$
  i.e., in terms of convolution $\varphi.f=f*\varphi^\vee$ (see below for the notation).
\end{Lem}
\begin{proof} Indeed, the function $x\longmapsto \int_G \varphi(y) f(xy)dy $
is in $I^1(G, U, \chi)$ since
\begin{align*}
 \int_{U\backslash G} \left|\int_G  \varphi(y) f(xy)dy\right|dx
\le \int_{U\backslash G} \int_G   \left|\varphi(y) f(xy)\right|dy dx
&=  \int_G   \left|\varphi(y)\right|  \left(\int_{U\backslash G}   \left| f(xy)\right|dx\right)dy\\
&\le ||\varphi||_\infty \text{vol}\left(\supp{(\varphi)}\right) ||f||_{U, 1}.
\end{align*}

For $g\in C_c^\infty(G)$, we have a continuous functional on $I^1(G, U, \chi)$ given by
$$
f\longmapsto \int_{U\backslash G} f(x)\left(\int_U g(ux)\chi(u) du\right)dx =\int_G f(x)g(x)dx.
$$
So, by definition of the action of $\varphi$, we have 
\begin{align*}
&\int_G \varphi.f(x)g(x)dx=\int_{U\backslash G}  \varphi.f(x) \left(\int_U g(ux)\chi(u) du\right) dx\\
  &=\int_G  \varphi(y) \left(\int_{U\backslash G} f(xy)\left(\int_U g(ux)\chi(u) du\right)dx  \right)dy\\
  &=\int_{U\backslash G} \left( \int_G\varphi(y) f(xy) dy\right) \left(\int_U g(ux)\chi(u) du\right)dx \\
&=  \int_{G}  \left(\int_G \varphi(y) f(xy)dy \right) g(x)dx
\end{align*}
which proves the desired formula for $\varphi.f$.
\end{proof}

\vskip .2in
Now, we recall  fundamental theorem of Harish--Chandra (\cite{hc}, Section 8, Theorem 1):

\begin{Lem}\label{cbr-4} Assume that $f\in  C^\infty(G)$ is
  $\cal Z(\mathfrak g_{\mathbb C})$--finite and $K$--finite on the right.
  Let $W\subset G$ be a neighborhood of $1$ invariant under conjugation of $K$. Then, there exists a function
  $\alpha\in C_c^\infty(G)$, $\supp{(\alpha)}\subset W$, and  invariant under
  conjugation of $K$ such that $f=f\star \alpha$.
\end{Lem}

We recall the formula for convolution
$$
f\star \alpha(x)=\int_G f(xy^{-1})\alpha(y)dy=\int_G f(xy)\alpha^\vee (y)dy=\alpha^\vee. f(x), 
$$
in our previous notation. Here $\alpha^\vee (x)=\alpha(x^{-1})$. Obviously, $\alpha^\vee$ is invariant
under the conjugation by $K$ if
$\alpha$ satisfies the same.

\vskip .2in 
Thus, as a corollary, we obtain the following standard result:

\begin{Cor}\label{cbr-5} Assume that  $f\in I^1(G, U, \chi)$, 
  is $\cal Z(\mathfrak g_{\mathbb C})$--finite,  and $K$--finite on the right.
  Then, $f$ is equal to a smooth function almost everywhere, therefore can be taken to be
  a smooth. Moreover, there exists $\beta\in C_c^\infty(G)$ invariant under the conjugation by $K$ such that $f=\beta.f$.
\end{Cor}
\begin{proof} Since $f$ is  $\cal Z(\mathfrak g_{\mathbb C})$--finite and $K$--finite on the right, 
  $f$ satisfies the same in the sense of
  distributions on $G$. By the usual theory it is then real analytic which proves the first claim.
  The second claim follows from above discussion.
  \end{proof}

\vskip .2in 
The following result is also standard (\cite{W1}, Corollary 3.4.7, Theorem 4.2.1 ). 

\begin{Lem}\label{cbr-6}. Assume that  $f\in I^1(G, U, \chi)$, 
  is $\cal Z(\mathfrak g_{\mathbb C})$--finite,  and $K$--finite on the right.
  Then, $(\mathfrak g, K)$--module generated by $f$ is admissible and it has a finite length.
\end{Lem}

\vskip .2in
Next, we prove the following lemma:

\begin{Lem}\label{cps-5} Assume  that $\varphi\in C^\infty(G)$ is  $\mathcal Z(\mathfrak g_{\mathbb C})$--finite, $K$--finite on 
the right, and in $L^1(G)$. Then, the following holds:
\begin{itemize}
\item[(i)] The integral $\int_U \varphi(ux) \chi(u) du$ converges absolutely for almost all $x\in G$. 
\item[(ii)] The function $x\longmapsto \int_U \varphi(ux) \chi(u) du$  is
  $\cal Z(\mathfrak g_{\mathbb C})$ and $K$--finite vector in the Banach representation  $I^1(G, U, \overline{\chi})$. 
\end{itemize}
\end{Lem}
\begin{proof} Let $f\in L^1(G)$. Then, 
$$
\int_{U\backslash G} \left(\int_U \left|f(ux) \right|du \right)dx=
\int_G \left|f(x) \right|dx<\infty.
$$
Thus, $f^U(x)\overset{def}{=}\int_U f(ux) \chi(u) du$ converges absolutely for almost all $x\in G$.  Obviously, we have 
$f^U\in I^1(G, U, \overline{\chi})$. Since also
$$
\int_{U\backslash G} \left|f^U(x)\right| dx\le 
\int_{U\backslash G} \left(\int_U \left|f(ux) \right|du \right)dx=
\int_G \left|f(x) \right|dx, 
$$
we obtain a continuous intertwining map $f\longmapsto f^U$ between  Banach representations 
of $G$: $L^1(G)\longrightarrow I^1(G, U, \overline{\chi})$ (given by right translations). The 
Garding space of $L^1(G)$, which by definition 
consists of functions $f\star \alpha$, $f\in L^1(G)$, $\alpha\in C_c^\infty(G)$, is mapped onto a (similarly defined) 
Garding space of $I^1(G, U, \overline{\chi})$. The restriction to the Garding space is $(\mathfrak g, K)$--equivariant map as 
one easily check. 

Above discussion shows (i) and $\varphi^U\in I^1(G, U, \overline{\chi})$. Also, Lemma \ref{cbr-4} shows that 
$\varphi$ is in the Garding space of  $L^1(G)$. Now, above discussion and Corollary \ref{cbr-5} imply (ii).
\end{proof} 

The existence of  $\cal Z(\mathfrak g_{\mathbb C})$--finite and
$K$--finite on the right  functions in $L^1(G)$ is standard and
explained in the proof of (\cite{MuicMathAnn}, Theorem 3-10).  They exists if and only if $G$ admits discrete series i.e., if
and only if $rank(G)=rank(K)$. Furthermore, if for example $U$ is the unipotent radical of a minimal parabolic subgroup of
$G$, then the integral $\int_U \varphi(ux) \chi(u) du$ converges absolutely for all $x\in G$
(\cite{W1}, Theorem 7.2.1) if for example $\varphi$ is a $K$--finite matrix coefficient of an integrable discrete series of $G$.

Where we use the following terms.  Let $(\pi, \mathcal H)$ be an irreducible unitary representation of $G$.
Let $(\ , \ )$ be the invariant inner product on $\mathcal H$. Let $\mathcal H_K$ be the space of $K$--finite inside
$\mathcal H$. A matrix coefficient of $\pi$ is a function on $G$ of the form $x\longmapsto (\pi(x)h, \ h')$, where
$h, h'\in\mathcal H$. Obviously, $\varphi\neq 0$ if and only if $h, h'\neq 0$. The matrix coefficient is $K$--finite on
the right (resp., on the left and on both sides)  if and only if $h\in \mathcal H_K$ (resp.,  $h'\in \mathcal H_K$ and
$h, h'\in \mathcal H_K$).

 \vskip .2in 
Let $\varphi$ be  any $K$--finite matrix coefficient of an integrable discrete series $\pi$. Then, we let 

\begin{equation}\label{FC-1}
\mathcal F_{(\chi, U)}^{loc}(\varphi)(x)\overset{def}{=} \int_U \varphi(ux)\overline{\chi(u)}du\ \  x\in G.
\end{equation}
By Lemma \ref{cps-5} (i), converges absolutely for almost all $x\in G$. By part (ii) of the same lemma and Corollary \ref{cbr-5},
is equal to a smooth function almost everywhere, therefore can be taken to be
a smooth. We assume that in what follows. The proof of Corollary \ref{cbr-5} show that it is
$\cal Z(\mathfrak g_{\mathbb C})$--finite and $K$--finite on the right. The meaning of (\ref{FC-1}) is contained in the following
result:

\begin{Cor} \label{gwm-7} Assume that  $\mathcal F_{(\chi, U)}^{loc}(\varphi)\neq 0$.  Then, there exists a closed
  irreducible  subscape of the Banach representation $I^1(G, U, \chi)$ which is infinitesimally equivalent to
  $(\pi, \mathcal H)$. 
\end{Cor}
\begin{proof} On can use the argument from the proof of Lemma \ref{cps-5}  to check that
  the map $\mathcal H_K\longrightarrow I^1(G, U, \chi)$ which maps $h'\in \mathcal H_K$ onto
  $x\longmapsto \int_U(\pi(ux)h', h) \overline{\chi(u)} du$ gives us the required infinitesimal equivalence.
\end{proof}
\section{On Certain Generic Representations of $G$}\label{gwm}

The goal of the present section is to put Corollary \ref{gwm-7} in a effective form exploring when
$(\chi, U)$--generic and integrable discrete series $\pi$ poses a $K$--finite matrix coefficient $\varphi$ such that
 $\mathcal F_{(\chi, U)}^{loc}(\varphi)\neq 0$. (See below for the definition of the terms used.) 

We start introducing some notation.
Let $\theta$ be the Cartan involution on $G$ with the fixed point equal to $K$. Let $\mathfrak g=\mathfrak k \oplus
\mathfrak p$ be the decomposition into $+1$ and $-1$--eigenspace of $d\theta$.  Let $\mathfrak a\subset \mathfrak p$ be the
maximal subspace subject to the condition that is Abelian Lie subalgebra. Let $A\subset G$ be the vector group with
Lie algebra $\mathfrak a$.

In this section we fix a minimal parabolic subgroup $P=MAN$ of $G$ in the usual way (see \cite{W1}, Section 2).
We have the Iwasawa decomposition $G=NAK$.

\vskip .2in

We recall the notion of a  norm on the group (see \cite{W1}, 2.A.2).  A norm $|| \ ||$ is a function
$G\longrightarrow [1, \infty[$ satisfying  the following properties:
    \begin{itemize}
    \item[(1)] $||x^{-1}||=||x||$, for all $x\in G$;
    \item[(2)]$||x\cdot y||\le ||x||\cdot || y||$, for all $x, y\in G$;
    \item[(3)] the sets $\left\{x\in G; \ \ ||x||\le r \right\}$ are compact  for all $r\ge 1$;
    \item[(4)] $||k_1\exp{(tX)} k_2||=||\exp{(X)}||^t$, for all $k_1, k_2\in K, X\in \mathfrak p, \ \ t\ge 0$.
    \end{itemize}
Any two norms $||\ ||_i$, $i=1,2$, are equivalent: there exist $C, r>0$ such that $||x||_1\le C ||x||^r_2$, for all $x\in G$.

Let $\Phi(\mathfrak g, \mathfrak a)$ be the set of all roots of $ \mathfrak a$ in $ \mathfrak g$. Let 
$\Phi^+(\mathfrak g, \mathfrak a)$ be the set of positive roots with respect to $\mathfrak n=Lie(N)$. Set

$$
\rho(H)= \frac12 \trace{\left(\text{ad}(H)|_{\mathfrak n}\right)}, \ \ H\in \mathfrak a.
$$
We set 
$$
m(\alpha)=\dim \ \mathfrak g_\alpha, \ \ \alpha\in \Phi^+(\mathfrak g, \mathfrak a).
$$
For $\mu\in \mathfrak a^\star$, we let 
$$
a^\mu = \exp{(\mu(H))}, \ \ a=\exp{(H)}.
$$
We define $A^+$ to be the set of all $a\in A$ such that $a^\alpha>1$ for all 
$\alpha\in \Phi^+(\mathfrak g, \mathfrak a)$.
Finally, we let
$$
D(a)=\prod_{\alpha\in \Phi^+(\mathfrak g, \mathfrak a)} \sinh{(\alpha(H))}^{m(\alpha)}, \ \  a=\exp{(H)}.
$$
Then, we may define a Haar measure on $G$ by the following formula:

$$
\int_G f(g) dg = \int_K \int_{A^+}\int_K D(a) f(k_1ak_2)dk_1 da dk_2, \ \  f\in C_c^\infty(G).
$$

\vskip .2in

Let $\left\{\alpha_1, \ldots, \alpha_r\right\}$ be the set of simple roots in
$ \Phi^+(\mathfrak g, \mathfrak a)$. Since $G$ is semisimple,
we have that this set spans $\mathfrak a^\star$. We define the dual basis $\left\{H_1, \ldots, H_r\right\}$ of $\mathfrak a$ 
in the standard way: $\alpha_i(H_j)=\delta_{ij}$. By (\cite{W1}, Lemma 2.A.2.3), there exists, $\mu, \eta\in  \mathfrak a^\star$
such that $\mu(H_i), \eta(H_j)>0$, for all $j$, and constants $C, D>0$ such that

$$
C a^\mu \le ||a||\le D a^\eta, \ \ a\in Cl(A^+).
$$

We remark that $\rho(H_j)>0$ for all $j$. So, we can find $c, d>0$ such that

$$
a^{c\rho} \le a^\mu  \le a^\eta \le a^{d\rho}, \ \ a\in Cl(A^+).
$$

We record this in the  next lemma:

\begin{Lem}\label{gwm-0} There exists real constants $c, C, d, D>0$ such that 
  $$
C a^{c\rho}  \le ||a||\le D a^{d\rho}, \ \ a\in Cl(A^+).
$$
\end{Lem}

\vskip .2in

Let $(\pi, \cal H)$ be a unitary irreducible representation of $G$  acting on the Hilbert space $\cal H$.
 We write $\langle \ , \ \rangle$ for the invariant  inner product on $\cal H$. We denote by $\cal H^\infty$ 
the subspace of smooth vectors in 
$\cal H$. It is a complete Fr\' echet space under the family of semi--norms:

$$
||h||_T =||d\pi(T)h||, \ \ T\in \cal U(\mathfrak g_{\mathbb C}),
$$
where $||\ ||$  is the norm on $\cal H$.

Let $U\subset G$ be a closed  subgroup. Let $\chi: U \longrightarrow \mathbb C^\times$ be a unitary character.
Consider a continuous functional $\lambda: \ \cal H^\infty\longrightarrow \mathbb C$
which satisfies 

\begin{equation}\label{gwm-00}
\lambda(\pi(u)h)=\chi(u)\lambda(h), \ \  u\in U, \ h\in \mathcal H^\infty.
\end{equation}
The fact that $\lambda$ is continuous means that there exists a constant $C>0$, and $T_1, \ldots, T_k \in
\cal U(\mathfrak g_{\mathbb C})$ such that 
$$
\left|\lambda(h)\right|\le C \left(\sum_{k=1}^l ||h||_{T_k}\right), \ \ h\in \mathcal H^\infty.
$$
By (\cite{W2}, Lemma 11.5.1), since $|\lambda|$ is a continuous semi--norm on $\cal H^\infty$, there exist
$d\in \mathbb R$ and a continuous semi--norm $\kappa$ on $H^\infty$ such that 

\begin{equation}\label{gwm-000}
|\lambda(\pi(g)h)|\le ||g||^d \kappa(h), \ \ g\in G, \ h\in \mathcal H^\infty.
\end{equation}

We define

\begin{equation}\label{gwm-0000}
d_{\pi, \lambda}=d_{\pi, \chi, U, \lambda}
\end{equation}
to be the infimum of all  $d\in \mathbb R$  for which there exist a continuous
semi--norm $\kappa$ on $\mathcal H^\infty$ such that (\ref{gwm-000}) holds. We remark that for each $d> d_{\pi, \lambda}$
there exists a continuous semi--norm $\kappa$ on $\mathcal H^\infty$ such that (\ref{gwm-000}) holds. Obviously, the function
$g\longmapsto \lambda(\pi(g)h)$ (for fixed $h\in \mathcal H_K$, $\neq 0$) is bounded provided that
$d_{\pi, \lambda}< 0$. 

\vskip .2in
We define the notion of generic representations that we are interested in. It generalizes the usual notion of
generic representation (i.e., the one having a Whittaker model) for quasi--split G \cite{kostant}. In all cases that we consider
$U$ is (conjugate) to a closed subgroup of $N$ (see Section \ref{app}).

\begin{Def}\label{gwm-1001} Let $U\subset G$ be a closed  subgroup and 
  $\chi: U \longrightarrow \mathbb C^\times$ be a unitary character. Let $(\pi, \cal H)$ be an
  unitary irreducible representation of $G$  acting on the Hilbert space
  $\cal H$. We say that  $(\pi, \cal H)$ is $(\chi, U)$--generic if there exists  a continuous functional
  $\lambda: \ \cal H^\infty\longrightarrow \mathbb C$ which satisfies
  $\lambda(\pi(u)h)=\chi(u)\lambda(h)$,  $u\in U$, $h\in \mathcal H^\infty$. 
\end{Def}

\vskip .2in
Now, we assume that $G$ admits discrete series and let
$\pi$ be an  integrable discrete series.  A complete characterization of them can be found in
\cite{milicic} (see also Section \ref{integrable}). Let $h\in \mathcal H_K$, $h\neq 0$. Consider the matrix coefficient given by
\begin{equation}\label{gwm-01}
\varphi(x)=(\pi(x)h, h).
\end{equation}
Since  $\pi$ is integrable and $\varphi$ is $K$--finite,
we have $\varphi\in L^1(G)$. This implies  that its complex conjugate
$\overline{\varphi}$ acts on any unitary representation of $G$.
In particular, it acts on $\pi$:
$$
\pi(\overline{\varphi})h{''}=\int_G \overline{\varphi(x)} \pi(x)h{''} dx, \ \ h{''}\in \mathcal H.
$$
We compute this action in the following standard and well--known  lemma. We have learned this lemma from
Savin which attributes the result to Mili\v ci\' c.

\begin{Lem}\label{gwm-1} (Mili\v ci\' c)
  Assume that $\pi$ is integrable discrete series. Let $d(\pi)>0$ be its formal degree. Then,
  $\pi(\overline{\varphi})$ is a rank-one operator given by
$$
\pi(\overline{\varphi})h''=\frac{1}{d(\pi)} (h'', h) h, \ \ h''\in \cal H.
$$
\end{Lem}
\begin{proof} The Schur orthogonality relation implies 
$$
\int_G (\pi(x)h, h) (h', \pi(x)h'') dx=\frac{1}{d(\pi)} (h, h'')(h', h), \ \ h', h''\in \cal H.
$$
Written differently
$$
(\pi(\overline{\varphi})h'', h')= \frac{1}{d(\pi)} (h'', h)(h, h').
$$
Equivalently
$$
\pi(\overline{\varphi})h''=\frac{1}{d(\pi)} (h'', h)h, \ \ h''\in \cal H.
$$
\end{proof}

\vskip .2in

We continue with a simple lemma:

\begin{Lem}\label{gwm-2} Let $\psi \in L^1(G)\cap C^\infty(G)$ be the  $\cal Z(\mathfrak g_{\mathbb C})$--finite and
  $K$--finite on the left function. Then, we have
  the following:
  \begin{itemize}
  \item[(i)] $\pi(\psi)\mathcal H\subset \mathcal H^\infty$.
  \item[(ii)] There exists a sequence $\psi_n \in C_c^\infty(G)$, $n\ge 1$, such that
    $l(T)\psi_n\overset{L^1}{\longrightarrow}
    l(T)\psi$, for all $T\in \mathcal U(\mathfrak g_{\mathbb C})$, where $l$ denotes the action derived from the action of
    $G$ by left translations.
  \end{itemize}
  \end{Lem}
\begin{proof} By the assumption on $\psi$, there exists $\alpha\in C_c^\infty(G)$ such that $\psi=\alpha\star \psi$
  (see \cite{hc}, Section 8, Theorem 1). Then, $\pi(\psi)=\pi(\alpha)\pi(\psi)$. So, we have
    $$
    \pi(\psi)\mathcal H =\pi(\alpha)\pi(\psi)\mathcal H\subset \pi(\alpha)\mathcal H\subset \mathcal H^\infty.
    $$
    This proves (i). By the standard measure theory, there exists a sequence $\eta_n \in C_c^\infty(G)$, $n\ge 1$, such that
    $\eta_n\overset{L^1}{\longrightarrow}\psi$. Then, for any $\beta\in  C_c^\infty(G)$, $\beta\star
    \eta_n\overset{L^1}{\longrightarrow}\beta \star \psi$. Also, for  $\beta\in  C_c^\infty(G)$ and $\eta\in L^1(G)$, the function
    $\beta\star \eta$ is in $C^\infty(G)$. Moreover, for  $T\in \mathcal U(\mathfrak g_{\mathbb C})$, we have by direct computation
    $$
    l(T)\left(\beta\star \eta\right)= (l(T)\beta)\star \eta.
    $$
    Finally, the sequence in (ii) can be taken to be  $\psi_n=\alpha\star \eta_n$, $n\ge 1$, as the reader can
    easily check. 
    \end{proof}

\vskip .2in 
 
  \begin{Lem}\label{gwm-3} Assume that $\pi$ is integrable and $(\chi, U)$--generic. Assume that there exists 
    $\lambda: \ \cal H^\infty\longrightarrow \mathbb C$ satisfying (\ref{gwm-00}) such that for some 
     $h\in \mathcal H_K$  we have 
    $$
\lambda(h)\neq 0, \ \ \text{and,} \ \ \int_G \left|\overline{\varphi(x)} \lambda(\pi(x)h)\right|dx <\infty,
   $$
$\varphi$ is given by (\ref{gwm-01}). 
    Then, $\mathcal F_{(\chi, U)}^{loc}(\varphi)\neq 0$. 
\end{Lem}
\begin{proof} We let
  $$
  \psi=\frac{d(\pi)}{(h, h)}\overline{\varphi}.
  $$ 
Then, by Lemma \ref{gwm-1}, we have
    $$
    \pi(\psi)h=h.
    $$
    Also,  $\psi$ satisfies the assumptions of Lemma \ref{gwm-2}, let us fix a sequence  $\psi_n \in C_c^\infty(G)$, $n\ge 1$,
    as in (ii) of that lemma. Then the sequence $\pi(\psi_n)h$, $n\ge 1$, converges to  $\pi(\psi)h=h$ in the Fr\' echet space
    $\mathcal H^\infty$ (by Lemma \ref{gwm-2} (ii)). Since, $\lambda$ is continuous we obtain the following:
    $$
    \lambda(h)=\lambda(\pi(\psi)h)=\lambda\left(\lim_n \pi(\psi_n)h\right)=\lim_n
    \lambda\left( \pi(\psi_n)h\right)=\lim_n \int_G \psi_n(x) \lambda(\pi(x)h)dx,
    $$
    where in the last step we have used the fact that the action of $C_c^\infty(G)$ on Hilbert space $\mathcal H$ and locally
    convex space $\mathcal H^\infty$ coincide which is easy to see by acting with continuous linear functional on
    $\mathcal H$ which are also continuous functionals on $\mathcal H^\infty$.  So, we get 

    $$
    \lim_n \int_G \psi_n(x) \lambda(\pi(x)h)dx=\lambda(h)\neq 0.
    $$
     
    We would like to interchange the limit and the integral. For this, we recall that Lemma \ref{gwm-2} (ii) implies that
    $\psi_n\overset{L^1}{\longrightarrow} \psi$. This implies that $\psi =(a.e.) \lim_n \psi_n$. Also, by our assumption, 
    the integral $\int_G \psi(x) \lambda(\pi(x)h)dx$ converges absolutely, and, therefore, we are  able
    to apply the Dominated convergence theorem.  We obtain
    $$
    \int_G \psi(x) \lambda(\pi(x)h)dx=\lambda(h)\neq 0.
    $$
    Combining with (\ref{gwm-00}),  again because of absolute convergence of the integral, this  implies
    $$
    \int_{U\backslash G} \left(\int_U\psi(ux)\chi(u)du\right) \lambda(\pi(x)h)dx=
     \int_G \psi(x) \lambda(\pi(x)h)dx=\lambda(h)\neq 0.
    $$
     Hence
     $$
     \int_U\psi(ux)\chi(u)du\neq 0, \ \ \text{(a.e.) $x\in G$}.
     $$
\end{proof}

\vskip .2in
We are left studying the absolute convergence of the integral  $\int_G \left|\overline{\varphi(x)} \lambda(\pi(ux)h)\right|dx$,
where $\varphi(x)=(\pi(x)h, h)$ with $h\in\mathcal H_k$, $h\neq 0$  as before. The function $\varphi$ is not just $L^1$ but by 
Mili\v ci\' c \cite{milicic} it satisfies the following
estimate:

\begin{Lem}\label{gwm-4} There exists  $\epsilon_\pi >0$ (see Definition \ref{integrable-2})
  which depends on equivalence class of $\pi$ only  such that  for each
    $\epsilon\in \ ]0, \ \epsilon_\pi]$ there exist $M>0$ and $k>0$ such that 
$$
  |\varphi(k_1ak_2)|\le M   a^{-(2+\epsilon)\rho}  \left(1+ \log{||a||}\right)^{k}, \ \
  k_1, k_2\in K, \ a\in Cl(A^+).
  $$
\end{Lem}
\begin{proof} By (\cite{milicic}, the theorem in the introduction),  for each
  $\epsilon\in \ ]0, \ \epsilon_\pi]$  there exist $M>0$ and $k>0$ such that 
$$
  |\varphi(x)|\le M \cdot \Xi(x)^{2+\epsilon} \left(1+ \log{||x||}\right))^k,
$$
where $\Xi$ is the zonal spherical function (\cite{W1}, 3.6). We recall (\cite{W1}, Theorem 4.5.3) the following
classical estimate:

$$
a^{-\rho} \le \Xi(a)\le E \cdot a^{-\rho}\left(1+ \log{||a||}\right)^l, \ \  a\in Cl(A^+),
$$
for some constants $E, l>0$. Thus, we obtain 

$$
  |\varphi(k_1ak_2)|\le M E^{2+\epsilon}  a^{-(2+\epsilon)\rho}  \left(1+ \log{||a||}\right)^{k+(2+\epsilon)l}, \ \
  k_1, k_2\in K, \ a\in Cl(A^+).
  $$
  This obviously proves the claim of the lemma.
\end{proof}

\vskip .2in
Now, we prove the following result:

\begin{Thm}\label{gwm-5} Assume that $\pi$ is integrable.
  Let   $\lambda: \ \cal H^\infty\longrightarrow \mathbb C$  be a non--zero continuous linear functional
  such that (\ref{gwm-00}) holds. Assume that $d_{\pi, \lambda}<\epsilon_\pi/d$, where constant $d>0$ is given by
  Lemma \ref{gwm-0}. Then, there exists a non--zero $K$--finite matrix
  coefficient $\varphi$ of $\pi$ such that $\mathcal F_{(\chi, U)}^{loc}(\varphi)\neq 0$
\end{Thm}
\begin{proof}  We keep the notation introduced in Lemma \ref{gwm-3} and immediately before the statement of Lemma \ref{gwm-4}.
  It is clear from Lemma \ref{gwm-3}  and comments after (\ref{gwm-0000}) that the lemma is true provided that
  $d_{\pi, \lambda} <0 $ since then the function $g\longmapsto \lambda(\pi(g)h)$  is bounded. So, assume that 
  $d_{\pi, \lambda} \ge 0$. By our assumption  $d\cdot d_{\pi, \lambda}<\epsilon_\pi$. We select any real number $\epsilon'\in
  \ ] d\cdot d_{\pi, \lambda}, \ \epsilon_\pi [$. Then, by definition of $d_{\pi, \lambda}$, we have 
$$
      |\lambda(\pi(g)h)|\le M'\cdot ||g||^{\epsilon'/d} , \ \ g\in G,
      $$
      for some constant $M'>0$. Hence, using the property (4) of the norm and  Lemma \ref{gwm-0}, there exists a
      constant $M{''}>0$
      such that
      $$
       |\lambda(\pi(k_1ak_2)h)|\le M{''}   a^{\epsilon'\rho}, \ \
  k_1, k_2\in K, \ a\in Cl(A^+).
  $$
  Select any $\epsilon\in  \ ] \epsilon', \ \epsilon_\pi [$. Then, by Lemma \ref{gwm-4}, 
      there exist $M>0$ and $k>0$ such that 
$$
  |\varphi(k_1ak_2)|\le M   a^{-(2+\epsilon)\rho}  \left(1+ \log{||a||}\right)^{k}, \ \
  k_1, k_2\in K, \ a\in Cl(A^+).
  $$
Thus, we obtain 

$$
\left|\varphi(k_1ak_2)\cdot \lambda(\pi(k_1ak_2)h)\right| \le M_1   a^{-2\rho + (\epsilon'-\epsilon)\rho}
\left(1+ \log{||a||}\right)^{k}, \ \
  k_1, k_2\in K, \ a\in Cl(A^+)
  $$
  for some constant $M_1>0$. Finally, by above recalled formula for Haar measure, using normalized measure on $K$ and
$$
D(a)\le a^{2\rho}, \ \ a\in A^+,
$$
we obtain

\begin{align*}
  \int_G \left|\overline{\varphi(g)}\lambda(\pi(g)h)\right| dg
  &= \int_K \int_{A^+}\int_K D(a)  \left|\varphi(k_1ak_2)\cdot \lambda(\pi(k_1ak_2)h)\right|
  dk_1 da dk_2\\
  &\le M_1 \int_{A^+} a^{(\epsilon'-\epsilon)\rho}
\left(1+ \log{||a||}\right)^{k} <\infty.
\end{align*}
\end{proof}

\vskip .2in
We expect that in all reasonable cases  i.e., $U\subset N$, the function $x\longmapsto \lambda(\pi(x)h)$ is bounded
for all $h\in \mathcal H^\infty$. This is what happens in the most important case.

\begin{Cor} \label{gwm-6} Assume that $\pi$ is integrable and $(\chi, N)$--generic for generic character
  $\chi$ (by definition, the differential $d\chi$ is non--trivial
  on any simple root subgroup $\mathfrak n_\alpha$, where $\alpha=\alpha_i$, $1\le i \le r$).  
  Then, there exists a non--zero $K$--finite matrix   coefficient $\varphi$ of $\pi$ such that
  $\mathcal F_{(\chi, U)}^{loc}(\varphi)\neq 0$.
\end{Cor} 
\begin{proof} This follows from Theorem \ref{gwm-5} as soon as we show that $x\longmapsto \lambda(\pi(x)h)$ is bounded
for all $h\in \mathcal H^\infty$.  
First, we define
  $$
\Lambda=  \Lambda_{\mathcal H_k}\in \mathfrak a^*
  $$
as in (\cite{W1}, 4.3). Since $\pi$ is in the discrete series, there exists  real numbers $k_1, k_2, \ldots, k_r> 0$ such that
$$
\Lambda=-\rho-\sum_{i=1}^r k_i\alpha_i,
$$
where  $\left\{\alpha_1, \ldots, \alpha_r\right\}$ is the set of simple roots in $ \Phi^+(\mathfrak g, \mathfrak a)$
as we denoted above (\cite{W1}, Proposition 5.1.3 and Theorem 5.5.4).
Furthermore, by (\cite{W2}, Lemma 15.2.3 and
Theorem 15.2.5), there exists a continuous seminorm $q_\Lambda$ on $\mathcal
H^\infty$ and a  real number $d\ge 0$  such that
$$
\left|\lambda(\pi(a)h')\right|\le a^\Lambda \left(1+ \log{||a||}\right)^d
q_\Lambda(h'), \ \ h'\in \mathcal H^\infty, \ a\in Cl(A^+).
$$
By above considerations, for $k\in K$ and $h'=\pi(k)h$ we have
$$
\left|\lambda(\pi(ak)h)\right|\le a^{-\rho} \left(1+ \log{||a||}\right)^d
q_\Lambda(\pi(k)h), \ \ a\in Cl(A^+)
$$
which implies that the function
$$
(a, k)\in A \times K \longmapsto \left|\lambda(\pi(ak)h)\right| \in \mathbb C
$$
is bounded on $Cl(A^+)\times K$.  But the analogous holds when $Q$ ranges over a finite set of  parabolic subgroups
with split component $A$: the function is bounded on $Cl(A^+_Q)\times K$  for
analogously defined $Cl(A^+_Q)$. But then the function itself is bounded
on $A\times K$ since $A$ is union of all $Cl(A^+_Q)$ when $Q$ ranges over all parabolic subgroups with a split
component $A$. This
is essentially the fact that $\mathfrak a$ is the union of closures of Weyl chambers.

Now, since $\chi$ is unitary, (\ref{gwm-00}) holds, the Iwasawa decomposition $G=NAK$ implies that
the function  $g \longmapsto \left|\lambda(\pi(g)h)\right|$ is bounded for all $h\in\mathcal H^\infty$.
To complete the proof, we apply Lemma \ref{gwm-3}.
  \end{proof}

\section{Preliminaries on Automorphic Forms and Poincar\' e Series}\label{paf}

From now on, in this paper we assume that $G$ is a group of $\mathbb R$--points of a
semisimple algebraic group $\mathcal G$ defined over $\mathbb Q$. Assume that $G$ is not compact and connected.
Let $\Gamma\subset G$ be congruence subgroup with respect to arithmetic structure given by the
fact that $\mathcal G$ defined over $\mathbb Q$ (see \cite{BJ}, or Section \ref{app}).  Then,  $\Gamma$ is a discrete subgroup
of $G$ and it has a finite covolume. 

An automorphic form (for $\Gamma$) is a function $f\in C^\infty(G)$ satisfying the following
three conditions (\cite{W3} or \cite{BJ}):

\begin{itemize}
\item[(A-1)] $f$ is  $\cal Z(\mathfrak g_{\mathbb C})$--finite and $K$--finite on the right;
\item[(A-2)] $f$ is left--invariant under $\Gamma$ i.e., $f(\gamma x)=f(x)$ for all $\gamma\in \Gamma$, $x\in G$;
\item[(A-3)] there exists $r\in\mathbb R$, $r>0$
  such that for each $u\in \mathcal U(\mathfrak g_{\mathbb C})$ there exists a constant
  $C_u>0$ such that $\left|u.f(x)\right|\le C_u \cdot ||x||^r$, for all $x\in G$.
\end{itemize}
We write $\mathcal A(\Gamma\backslash G)$ for the vector space of all automorphic forms. It is easy to see that
$\mathcal A(\Gamma\backslash G)$ is a $(\mathfrak g, K)$--module. An automorphic form
$f\in \mathcal A(\Gamma\backslash G)$  is a cuspidal automorphic form if for every proper $\mathbb Q$--proper
parabolic $\mathcal P\subset \mathcal G$ we have
$$
\int_{U\cap \Gamma \backslash U} f(ux)dx=0, \ \ x\in G,
$$
where $U$ is the group of $\mathbb R$--points of the unipotent radical of  $\mathcal P$.
We remark that the quotient $U\cap \Gamma \backslash U$ is compact. We use normalized $U$--invariant measure on
$U\cap \Gamma \backslash U$. The space of all cuspidal automorphic forms for $\Gamma$
is denoted by $\mathcal A_{cusp}(\Gamma\backslash G)$.

\vskip .2in 

For the sake of completeness, we state the following standard result:

\begin{Lem}\label{paf-1} Under above assumptions, we have the following:
  \begin{itemize}
         \item[(a)] If  $f\in C^\infty(G)$ satisfies (A-1), (A-2), and
        there exists  $r\ge 1$ such that $f\in L^r(\Gamma\backslash G)$, then f satisfies (A-3), and it is therefore an
        automorphic form. We speak about $r$--integrable automorphic form, for $r=1$ (resp., $r=2$) we speak about integrable
        (resp., square--integrable) automorphic form.
      \item[(b)] Let $r\ge 1$. Every $r$--integrable automorphic form is integrable.
      \item[(c)] Bounded integrable automorphic form is square--integrable.
      \item[(d)] If $f$ is square integrable automorphic form, then the minimal $G$--invariant closed subspace of
        $L^2(\Gamma\backslash G)$ is a direct is of finitely many irreducible unitary representations.
       \item[(e)] Every cuspidal automorphic form is square--integrable.
    \end{itemize}
  \end{Lem}
\begin{proof} For the claims (a) and (e) we refer to \cite{BJ} and reference there.
  Since the volume of $\Gamma \backslash G$ is finite, the claim (b) follows from H\" older inequality (as in
  \cite{MuicMathAnn}, Section 3). The claim (c) is obvious. The claim (d) follows from (\cite{W1}, Corollary 3.4.7 and
  Theorem  4.2.1).
\end{proof}

\vskip .2in
In the generality that we consider in this section, the constant function is one example of square--integrable automorphic form.
The other one are provided by the standard Poincar\' e Series (see \cite{bb}, Theorem 5.4,
and \cite{Borel-SL(2)-book}, Theorem 6.1). The case of $SL_2(\mathbb R)$ is sharpen in (\cite{MuicJNT}, Lemma 2.9).
The adelic version  of the following lemma (i) and (ii)  is contained  in (\cite{MuicMathAnn}, Theorem 3.10).

\vskip .2in 

\begin{Lem}\label{cps-2} Assume  $\varphi\in C^\infty(G)$ that is  $\cal Z(\mathfrak g_{\mathbb C})$--finite,
  $K$--finite, and in $L^1(G)$.  Then, we have the following:
  \begin{itemize}
    \item[(i)] The   Poincar\' e series $P(\varphi)(g)=\sum_{\gamma\in \Gamma} \varphi(\gamma g)$ converges absolutely and
  uniformly on compact sets, and it is a cuspidal automorphic form.

\item[(ii)] (Non--vanishing criterion of \cite{MuicMathAnn})  Let $\pi$ be an integrable discrete series and
  $\varphi$ be a non--zero $K$--finite matrix coefficient of $\pi$.  Assume that there exists a compact neighborhood $C$
  in $G$ such that 
$$
\int_C |\varphi(g)|dg > \int_{G-C}
|\varphi(g)|dg\  \text{ and } \ \Gamma \cap C\cdot C^{-1}=\{1\}.
$$
Then, $P(\varphi)\neq 0$.

\item[(iii)]   Let $\Gamma_1\supset \Gamma_2\supset \dots$ be a sequence
  of discrete subgroups of $G$ such that   $\cap_{n\ge 1}
  \Gamma_n=\{1\}$.  Let $\pi$ be an integrable discrete series and
  $\varphi$ be a non--zero $K$--finite matrix coefficient of $\pi$. Put $P_n(\varphi)(x)=\sum_{\gamma\in \Gamma_n}
  \varphi (\gamma x)$, for all $n\ge 1$. Then, there exists $n_0$ depending on $\varphi$ such that
   $P_n(\varphi)\neq 0$ for $n\ge n_0$.
    \end{itemize}
\end{Lem}
\begin{proof} (i) has a proof similar to that of (\cite{Borel-SL(2)-book}, Theorem 6.1) where the case $G=SL_2(\mathbb R)$
  is considered (see also the proof of (\cite{MuicMathAnn}, Theorem 3.10) in adelic settings).
The  claim (ii) is (\cite{MuicMathAnn}, Theorem  4-1). Finally, the claim (iii) is (\cite{MuicMathAnn}, Corollary 4-9).
\end{proof}

\section{On a Construction of Certain Poincar\' e Series}\label{cps}

In this section we maintain assumptions of  Section \ref{paf}.  We assume that $U$ is the group of $\mathbb R$--points of the
unipotent radical of a proper $\mathbb Q$--proper parabolic $\mathcal P\subset \mathcal G$. Then, $U\cap \Gamma$ is
cocompact in $U$. Let $\chi$ be a character of $U$ trivial on $U\cap \Gamma$.

The non--vanishing of Poincar\' e series discussed in Lemma \ref{cps-2} is a subtle problem
(see applications of sufficient criterion Lemma \ref{cps-2} (vi) in adelic settings \cite{MuicMathAnn} and in
the case of $SL_2(\mathbb R)$ \cite{MuicJNT}). In this section, we discuss  a different approach based on
Fourier coefficients.

\begin{Def}\label{cps-1000}
The $(\chi, U)$--Fourier coefficient of a function 
$f\in C^\infty(\Gamma\backslash G)$  is defined as follows:

$$
\mathcal F_{(\chi, U)}(f)(x)=\int_{U\cap \Gamma\backslash U} f(ux)\overline{\chi(u)}du,
$$
where we use a normalized measure on a compact topological space $U\cap \Gamma\backslash U$.
We say that an automorphic form $f$ (see Section \ref{paf} for definition) is $(\chi, U)$--generic if
$F_{(\chi, U)}(f)\neq 0$. A $(\mathfrak g, K)$--submodule $V\subset \mathcal A(\Gamma\backslash G)$ is generic
if there exists at least one $f\in W$ such that $F_{(\chi, U)}(f)\neq 0$.
\end{Def}

\vskip .2in
Now, we compute the Fourier coefficients of Poincar\' e series described by Lemma \ref{cps-2}. We remind the reader that
$\varphi^\vee(x)=\varphi(x^{-1})$, and $l(x)\varphi(y)=\varphi(x^{-1}y)$ is a left translation.

\begin{Lem}\label{cps-3} Assume  $\varphi\in C^\infty(G)$ that is  $\cal Z(\mathfrak g_{\mathbb C})$--finite,
  $K$--finite, and in $L^1(G)$. Then, the
  Fourier coefficient $\mathcal F_{(\chi, U)} \left(P(\varphi)\right)$ of the Poincar\' e series $P(\varphi)$ is given by the
  following expression:

  $$
  \mathcal F_{(\chi, U)}
  \left(P(\varphi)\right)(x)= \sum_{\gamma\in  U\cap \Gamma \backslash \Gamma }  \int_{ U}l(x)\varphi^\vee (u\gamma)\chi(u)du, \ \ x\in G.
$$
\end{Lem}
\begin{proof} Since the series $P(\varphi)(g)=\sum_{\gamma\in \Gamma} \varphi(\gamma g)$ converges absolutely and
  uniformly on compact sets (see Lemma \ref{cps-2} (i)),   and $U\cap \Gamma\backslash U$ is compact, we have a
  standard unfolding

  \begin{align*}
    \mathcal F_{(\chi, U)} \left(P(\varphi)\right)(x)&=
    \int_{U\cap \Gamma\backslash U} \left(\sum_{\gamma\in \Gamma} \varphi(\gamma ux)\right)\overline{\chi(u)}du\\
    &= \int_{U\cap \Gamma\backslash U} \left(\sum_{\gamma\in \Gamma \setminus U\cap \Gamma} \ \ \sum_{\delta \in  U\cap \Gamma} \varphi(\gamma\delta ux)
    \right)\overline{\chi(u)}du\\
    &= \sum_{\gamma\in \Gamma \setminus U\cap \Gamma} \ \ 
    \int_{U\cap \Gamma\backslash U} \left( \sum_{\delta \in  U\cap \Gamma} \varphi(\gamma\delta ux)\right)\overline{\chi(u)}du.\\
\end{align*}

Since $\chi$ is trivial on $U\cap \Gamma$, the integral inside the sum can be written in the form

\begin{align*}
\int_{U\cap \Gamma\backslash U} \left( \sum_{\delta \in  U\cap \Gamma} \varphi(\gamma\delta ux)\right)\overline{\chi(u)}du &=
\int_{U\cap \Gamma\backslash U} \left( \sum_{\delta \in  U\cap \Gamma} \varphi(\gamma\delta ux)\right)\overline{\chi(\delta u)}du\\
&=\int_{ U}\varphi(\gamma ux)\overline{\chi(u)}du\\
&=\int_{ U}l(x)\varphi^\vee (u^{-1}\gamma^{-1})\overline{\chi(u)}du\\
&=\int_{ U}l(x)\varphi^\vee (u\gamma^{-1})\chi(u)du.\\
\end{align*}

Finally, we have 
\begin{align*}
\mathcal F_{(\chi, U)} \left(P(\varphi)\right)(x)&=  \sum_{\gamma\in \Gamma \setminus U\cap \Gamma}  \int_{ U}l(x)\varphi^\vee 
(u\gamma^{-1})\chi(u)du\\
&=  \sum_{\gamma\in  U\cap \Gamma \backslash \Gamma }  \int_{ U}l(x)\varphi^\vee (u\gamma)\chi(u)du.\\
\end{align*}
\end{proof}

\vskip .2in 
Lemma \ref{cps-3} suggest that we consider the following class of automorphic forms:

\begin{Lem}\label{cps-7} Assume  $\varphi\in C^\infty(G)\cap I^1(G, U, \chi)$ (see Section \ref{cbr} for notation)
 is  $\cal Z(\mathfrak g_{\mathbb C})$--finite and $K$--finite on 
the right. Then, the series 
$$
W_{(\chi, U)}(\varphi)(x)=W^\Gamma_{(\chi, U)}(\varphi)(x)\overset{def}{=}\sum_{\gamma \in U\cap \Gamma\backslash \Gamma} \varphi(\gamma x)
$$
converges absolutely and uniformly on compacta in $G$ to  an integrable automorphic form
i.e., $W_{(\chi, U)}(\varphi) \in \mathcal A(\Gamma \backslash G)\cap L^1(\Gamma \backslash G)$.
\end{Lem}
\begin{proof}The idea of the proof that it converges absolutely on compacta in $G$
is based on a slight improvement of classical argument as explained in (\cite{MuicMathAnn}, Theorem 3-10). We adapt the 
argument to the present situation. 

Since $\Gamma$ is discrete, there exists a compact neighborhood $W$ of $1\in G$ such that

\begin{equation}\label{cps-8}
  \Gamma \cap W=\{1\}.
\end{equation}
To prove that the series converges absolutely and uniformly on compacta in $G$, it is enough to prove that for each 
$g\in G$ there exists a  compact neighborhood $g\in D\subset G$ such that the series converges absolutely and uniformly. 
So, let $g\in G$ be fixed. We select and compact neighborhoods  $g\in D\subset G$ and  $1\in V\subset G$ such that 
$DV^{-1} V D^{-1}\subset W$. By (\ref{cps-8}),  we have the following:
\begin{equation}\label{cps-9}
  \text{if $\left(U\cap \Gamma\right)\gamma DV^{-1} \cap \left(U\cap \Gamma\right)
    \delta DV^{-1}  \neq \emptyset$, for $\gamma, \delta\in \Gamma$, then $\gamma\in  
\left(U\cap \Gamma\right)\delta$.}
\end{equation}
Since $\varphi$ is  $\cal Z(\mathfrak g_{\mathbb C})$--finite and $K$--finite on 
the right, there exists $\alpha\in C_c^\infty(G)$, $\supp{(\alpha)} \subset V$, such that 
$\varphi=\varphi\star \alpha$ (see Lemma \ref{cbr-4}). This can be rewritten as follows:
$$
\varphi(x)=\varphi\star \alpha(x)= \int_{G}
\varphi(xy^{-1})\alpha(y)dy=\int_{V^{-1}}
\varphi(xy)\alpha^\vee(y)dy,
$$
where as usual, we let $\alpha^\vee(y)=\alpha(y^{-1})$. Now, since we have the obvious equality of $\sup$--norms
$|\alpha^\vee |_\infty=|\alpha|_\infty$,  for $x\in D$ and $\gamma\in \Gamma$, we have the 
following:

\begin{align*}
|\varphi(\gamma x)|& \leq |\alpha|_\infty\cdot 
\int_{V^{-1}} |\varphi(\gamma xy)|dy=
|\alpha|_\infty\cdot \int_{\gamma x V^{-1}}
|\varphi(y)|dy \\ 
&\leq 
|\alpha|_\infty\cdot \int_{\gamma D V^{-1}}
|\varphi(y)|dy=
|\alpha|_\infty\cdot \int_G 1_{\gamma D V^{-1}}(y)|\varphi(y)|dy\\
&= |\alpha|_\infty\cdot 
\int_{U\cap \Gamma \backslash G} \left(\sum_{\delta\in U\cap \Gamma}  1_{\gamma D V^{-1}}(\delta y)|\varphi(\delta y)|
\right)dy\\
&= |\alpha|_\infty\cdot 
\int_{U\cap \Gamma \backslash G} |\varphi(y)| \left(\sum_{\delta\in U\cap \Gamma}  1_{\gamma D V^{-1}}(\delta y) \right)dy\\
&= |\alpha|_\infty\cdot 
\int_{U\cap \Gamma \backslash \left(U\cap \Gamma\right)\gamma D V^{-1}} 
|\varphi(y)| \left(\sum_{\delta\in U\cap \Gamma}  1_{\gamma D V^{-1}}(\delta y) \right)dy\\
&= |\alpha|_\infty\cdot 
\int_{U\cap \Gamma \backslash \left(U\cap \Gamma\right)\gamma D V^{-1}}  |\varphi(y)| \left(\sum_{\substack{\delta\in U\cap \Gamma\\
    \delta\in \gamma D V^{-1}y^{-1}}} 1 \right)dy\\
\end{align*}
\begin{align*}
&\le |\alpha|_\infty\cdot 
\int_{U\cap \Gamma \backslash \left(U\cap \Gamma\right)\gamma D V^{-1}}  |\varphi(y)| \left(\sum_{\substack{\delta\in U\cap \Gamma\\
\delta\in \gamma D V^{-1}VD^{-1}\gamma^{-1}}} 1 \right)dy\\
&\le |\alpha|_\infty\cdot 
\int_{U\cap \Gamma \backslash \left(U\cap \Gamma\right)\gamma D V^{-1}}  |\varphi(y)| \left(\sum_{\substack{\delta\in U\cap \Gamma\\
\delta\in \gamma W \gamma^{-1}}} 1 \right)dy\\
&=|\alpha|_\infty\cdot 
\int_{U\cap \Gamma \backslash \left(U\cap \Gamma\right)\gamma D V^{-1}}  |\varphi(y)|dy.\\
\end{align*}
The last equality holds since
$\delta\in \gamma W \gamma^{-1}$ implies that $\gamma^{-1}\delta\gamma\in W\in\Gamma=\{1\}$ by (\ref{cps-8}).
  
In view of (\ref{cps-9}) and the fact that $U\cap \Gamma$ is cocompact in $U$, we get
$$
 |\alpha|_\infty \cdot \sum_{\gamma\in U\cap\Gamma \backslash \Gamma}\int_{U\cap \Gamma \backslash \left(U\cap \Gamma\right)\gamma D V^{-1}}
|\varphi(y)|dy\le  |\alpha|_\infty\cdot 
\int_{U\cap \Gamma \backslash G}  |\varphi(y)|dy <\infty
$$
which  shows that the series converges absolutely and uniformly on $D$.

Since
$$
\int_{\Gamma \backslash G} \left(\sum_{\gamma\in  U\cap \Gamma \backslash \Gamma } 
\left|\varphi (\gamma x)\right| \right) dx= 
\int_{U\cap \Gamma\backslash G} \left|\varphi(x)\right|dx <\infty,
$$
we have  $W_{(\chi, U)}(\varphi) \in L^1(\Gamma\backslash G)$. The argument used in the proof of
(\cite{MuicMathAnn}, Theorem 3-10) shows that 
$\cal Z(\mathfrak g_{\mathbb C})$--finite and $K$--finite on the right. So, it is an
automorphic form by Lemma \ref{paf-1} (b).
\end{proof}

\vskip .2in

Now, we come to the main result of the present section.

\begin{Prop}\label{cps-10}  Let $\pi$ be an integrable discrete series of $G$. 
  Assume that there exists  a $K$--finite  matrix coefficient
  $\varphi$ of $\pi$ such that the following holds:
  $$
  W^\Gamma_{(\chi, U)}\left(\mathcal F_{(\chi, U)}^{loc}(\varphi)\right)\neq 0.
  $$
Then, there is a realization of  $\pi$  as a $(\chi, U)$--generic cuspidal  automorphic  representation.
\end{Prop}
\begin{proof} By the assumption 
  
\begin{itemize}
\item[(i)]  $\psi\overset{def}{=}\mathcal F_{(\overline{\chi}, U)}^{loc}(\overline{\varphi})\neq 0$, and 
\item[(ii)] $W^\Gamma_{(\overline{\chi}, U)}(\psi)\neq 0$.
\end{itemize}

By the comment in the paragraph containing (\ref{FC-1}), we may apply
Lemma \ref{cps-7}  to construct automorphic form 
$W_{(\overline{\chi}, U)}(\psi)$. In particular,  $W_{(\overline{\chi}, U)}(\psi)$ is real--analytic. So, we have the Taylor expansion 
$$
W_{(\overline{\chi}, U)}(\psi)\left( \exp{(X)}\right)=\sum_{n=0}^\infty \frac{1}{n!} X^n.
W_{(\overline{\chi}, U)}(\psi)\left(1\right), \ \ 
$$
where $X$ belongs to a convenient neighborhood $\mathcal V$ of $0$ in $\mathfrak g$. Since $G$ is assumed to be connected,
$W_{(\overline{\chi}, U)}(\psi)\neq  0$ if and only if 
there exists  $X\in \mathcal V$, and $n\ge 1$ such that  
\begin{equation}\label{cps-11}
 X^n.W_{(\overline{\chi}, U)}(\psi)\left(1\right)\neq 0.
\end{equation}

Let us fix $X\in \mathfrak g$, and $n\ge 1$ satisfying (\ref{cps-11}), and make this condition explicit.
 Let us write $(\ , \ )$ for the invariant inner product on the Hilbert space $\mathcal H$ 
realizing the representation $\pi$. Let us write $\mathcal H_K$ for the set of $K$--finite vectors in $\mathcal H$. 

For $v, v'\in \mathcal H_K$, we let  
$$
\varphi_{v, v'}(g)=(\pi(g)v, \ v'), \ \ g\in G.
$$
Then
$$
\overline{\varphi}_{v, v'}(g)=(v', \ \pi(g)v), \ \ g\in G.
$$
Furthermore, for $X\in \mathfrak g$, $g\in G$, we have 
\begin{equation}\label{cps-100000}
X.\overline{\varphi}_{v, v'}(g)=\frac{d}{dt}|_{t=0} (v', \ \pi(g)\pi(\exp{(tX)})v)=
(v', \ \pi(g)d\pi(X)v)=\overline{\varphi}_{ d\pi(X)v, v'}(g).
\end{equation}
We also let (see (\ref{FC-1}))
$$
\psi_{v, v'}(g)=\mathcal F_{(\overline{\chi}, U)}^{loc}(\overline{\varphi}_{v, v'})(g)=
\int_U\overline{\varphi}_{v, v'}(ug)\chi(u)du, \ \ g\in G.
$$
Then, by (\ref{cps-100000}) and the discussion in the proof of Lemma\ref{cps-5},  we have 
$$
\psi_{ d\pi(X)v, v'}=X.\psi_{v, v'}, \ \ v, v'\in \mathcal H_K, \ X\in \mathfrak g.
$$
Next, by our assumption, there exists $h, h'\in \mathcal H_K$ such that 
$$
\varphi=\varphi_{h, h'}.
$$
Then, (\ref{cps-11}) implies
\begin{align*}
 W_{(\overline{\chi}, U)}(\psi_{d\pi(X)^nh, h'})\left(1\right)&= W_{(\overline{\chi}, U)}(X^n.\psi_{h, h'})\left(1\right)\\
&=X^n.W_{(\overline{\chi}, U)}(\psi_{h, h'})\left(1\right)\\
&=X^n.W_{(\overline{\chi}, U)}(\psi)\left(1\right)\neq 0.
\end{align*}
We remark that one can see that $W_{(\overline{\chi}, U)}$ commutes with the action of $\mathcal U(\mathfrak g_{\mathbb C})$  as in the
proof of (\cite{MuicMathAnn}, Theorem 3-10). Thus, by Lemma \ref{cps-3}, we have 

\begin{align*}
\mathcal  F_{(\chi, U)} \left(P(\overline{\varphi}^{\vee}_{d\pi(X)^nh, h'})\right)(1) &=W_{(\overline{\chi}, U)}\left(
\mathcal F^{loc}_{(\overline{\chi}, U)}(\overline{\varphi}_{d\pi(X)^nh, h'})\right)(1)\\
&=
 W_{(\overline{\chi}, U)}(\psi_{d\pi(X)^nh, h'})\left(1\right)\neq 0.
\end{align*}
But
$$
\overline{\varphi}^{\vee}_{d\pi(X)^nh, h'}(x)=\overline{(\pi(x^{-1})d\pi(X)^nh, h')}=(\pi(x)h', d\pi(X)^nh),
$$
is a $K$--finite matrix coefficient of $\pi$. Now, we use Lemma \ref{cps-2} to complete the proof.
\end{proof}

\vskip .2in
In the case of $U=N$ and $\chi$ is a generic character, the condition (i) is fulfilled  by at least one matrix coefficient
by Corollary  \ref{gwm-6}.

\vskip .2in
We end the following section with the following proposition.

\begin{Prop}\label{cps-100} We assume that
  $U\cap \Gamma \backslash U$ is Abelian. Assume that $\pi$ is an integrable discrete series of $G$. Then, if there exists a
  $K$--finite matrix coefficient $\varphi$ of $\pi$ such that $P(\varphi)\neq 0$, then $\pi$ is infinitesimally equivalent
  to a closed subrepresentation of $I^1(G, U, \chi)$ for some $\chi\in \widehat{\left(U\cap \Gamma\backslash U\right)}$.
  \end{Prop}
\begin{proof} Since $U\cap \Gamma\backslash U$ is a compact Abelian group, we can compute the Fourier expansion of $P(\varphi)$.
We recall that we have normalized the Haar measure on $U\cap \Gamma\backslash U$ such that the total volume is equal to one.
In $L^2\left(U\cap \Gamma\backslash U\right)$, we have the following Fourier expansion:
  $$
P(\varphi)(ux)=\sum_{\chi\in \widehat{\left(U\cap \Gamma\backslash U\right)}}c_\chi\left(P(\varphi)\right)(x)\chi(u), \ \ x\in G,
  $$
where the Fourier coefficients are given by
$$
c_\chi\left(P(\varphi)\right)(x)=\int_{U\cap \Gamma\backslash U} P(\varphi)(ux)\overline{\chi(u)} du=
\mathcal F_{(\chi, U)}  \left(P(\varphi)\right)(x).
$$

Now, if $P(\varphi)\neq 0$, then the Fourier coefficient $\mathcal F_{(\chi, U)}  \left(P(\varphi)\right)\neq 0$, for some
$\chi\in \widehat{\left(U\cap \Gamma\backslash U\right)}$.
But this Fourier coefficient is real--analytic function of $x\in G$. Hence, since $G$ is connected, as in the proof
of Proposition \ref{cps-10}, there  exists $X\in \mathfrak g$, and $n\ge 1$ such that
$$
X^n.\mathcal F_{(\chi, U)}\left(P(\varphi)\right)(1)\neq 0. 
$$
We can write this as follows:
$$
\mathcal F_{(\chi, U)}\left(P(X^n.\varphi)\right)(1)=X^n.\mathcal F_{(\chi, U)}\left(P(\varphi)\right)(1)\neq 0.
$$
Using Lemma \ref{cps-3}, this can be written as follows:
$$
W_{(\overline{\chi}, U)}\left(\mathcal F^{loc}_{(\overline{\chi}, U)}\left(\left(X^n.\varphi\right)^\vee\right)\right)(1)=
\mathcal F_{(\chi, U)}\left(P(X^n.\varphi)\right)(1)
\neq 0.
$$
This implies 
$$
\mathcal F^{loc}_{(\chi, U)}\left(\overline{\left(X^n.\varphi\right)}^\vee\right)\neq 0.
$$
is not equal to zero almost everywhere. As in the proof of Proposition \ref{cps-10}, we can check that
$\overline{\left(X^n.\varphi\right)}^\vee$ is a  $K$--finite matrix coefficient of $\pi$. Hence,
$\pi$ is infinitesimally equivalent to a closed subrepresentation of $I^1(G, U, \chi)$ by Corollary \ref{gwm-7}. 
\end{proof}

\section{Applications}\label{app}

We maintain the assumptions of Sections \ref{paf} and \ref{cps} (see the first paragraphs in those sections). 
We need to study when  $W_{(\chi, U)}(\varphi)\neq 0$ for $\varphi\in C^\infty(G)\cap I^1(G, U, \chi)$ 
 which is  $\cal Z(\mathfrak g_{\mathbb C})$--finite and $K$--finite on 
the right (see  Lemma \ref{cps-7}).
In this case the criterion Lemma \ref{cps-2} (ii) (see (\cite{MuicMathAnn}, Theorem 4-1)) is not applicable here.
We use a generalization of (\cite{MuicMathAnn}, Theorem 4-1)  given by (\cite{MuicIJNT}, Lemma 2-1):

\begin{Lem}
\label{nv-thm} Let $G$ be a locally compact
  unimodular   (Hausdorff) topological group.
Let $\Gamma\subset G$ be its discrete subgroup and
$\Gamma_1\subset \Gamma$ an arbitrary subgroup. 
We let $\eta:\Gamma\rightarrow \bbC^\times$ be an unitary character of 
$\Gamma$ trivial on $\Gamma_1$. 
Let $\varphi\in L^1(\Gamma_1\backslash G)$.   Assume that there exists
a subgroup $\Gamma_2\subset \Gamma$ such that $\Gamma_1$ is normal
subgroup of $\Gamma_2$ of finite index and  there exists a compact set $C$
in $G$ such that the following conditions hold:
\begin{itemize}
\item[(1)] $\varphi(\gamma g)=\overline{\eta(\gamma)}\varphi(g)$, for all
  $\gamma\in \Gamma_2$ and almost all $g\in G$, 
\item[(2)]  $\Gamma \cap C\cdot C^{-1}\subset \Gamma_2$, and  
\item[(3)] $\int_{\Gamma_1\backslash \Gamma_1\cdot C}
\left|\varphi(g)\right|dg > \frac{1}{2}
\int_{\Gamma_1\backslash G} \left|\varphi(g)\right|dg$. 
\end{itemize}
Then the Poincar\' e series defined by 
$$
P^{(\chi)}_{\Gamma_1\backslash\Gamma}(\varphi)(g)\overset{def}{=}
\sum_{\gamma\in \Gamma_1\backslash\Gamma}\eta(\gamma) \varphi(\gamma\cdot
g)
$$ 
converges absolutely almost everywhere to a non--zero element of  $L^1(\Gamma \backslash G)$.
\end{Lem}

In our case, the lemma immediately implies
the following result:
 
\begin{Lem} \label{cps-12}  Assume  that
$\varphi\in C^\infty(G)\cap I^1(G, U, \chi)$ is  
$\cal Z(\mathfrak g_{\mathbb C})$--finite and $K$--finite on 
the right. Then, the series (see Lemma \ref{cps-7}) 
$$
W_{(\chi, U)}(\varphi)(x)=\sum_{\gamma \in U\cap \Gamma\backslash \Gamma} \varphi(\gamma x)
$$
is non--zero provided that there exists a compact set $C\subset G$ such that 
\begin{itemize}
\item[(1)]  $\Gamma \cap C\cdot C^{-1}\subset  U\cap \Gamma$,  and  
\item[(2)] $\int_{\left( U\cap \Gamma\right)\backslash \left( U\cap \Gamma\right)\cdot C}
\left| \varphi( g)\right|dg > \frac{1}{2}
\int_{\left( U\cap \Gamma\right)\backslash G}    \left| \varphi(g)\right|dg$. 
\end{itemize}
\end{Lem}
\begin{proof} Since $U\cap\Gamma$ is cocompact in $U$ and $\chi$ is trivial on $U\cap \Gamma$, we see that
  $\varphi\in L^1(U\cap\Gamma\backslash G)$. Now, we apply Lemma \ref{nv-thm}.
\end{proof}

\vskip .2in
We now introduce some congruence subgroups. Let $\mathbb A$ (resp., $\mathbb A_f$) be the ring of adeles (resp., finite adeles) of
$\mathbb Q$.  For each prime $p$, let 
$\mathbb Z_p$ be the maximal compact subring of $\mathbb Q_p$. Then, for almost all primes $p$, the group $\mathcal G$
is unramified over  $\mathbb Q_p$; in particular, $\mathcal G$ is a group scheme over $\mathbb Z_p$, and   $\mathcal G(\bbZ_p)$
is a hyperspecial maximal compact subgroup of  $\mathcal G(\mathbb Q_p)$ (\cite{Tits}, 3.9.1).  Let $\mathcal G(\bbA_f)$
be the restricted product of all groups $G(\mathbb Q_p)$ with respect to the groups $\mathcal G(\mathbb Z_p)$ where $p$ ranges
over all primes $p$ such that $\mathcal G$ is unramified over $\mathbb Q_p$:
$$
\mathcal G(\mathbb A_f)=\prod'_p \mathcal G(\mathbb Q_p).
$$
Note that
$$
\mathcal G(\mathbb A)= \mathcal G(\mathbb R)\times \mathcal G(\mathbb A_f).$$
The group $\mathcal G(\mathbb Q)$ is embedded into $\mathcal G(\mathbb R)$ and
$\mathcal G(\mathbb Q_p)$. It is embedded diagonally in $\mathcal G(\mathbb A_f)$ and
in $\mathcal G(\mathbb A)$.

\vskip .2in
The congruence subgroups of $\mathcal G$ are defined as follows (see \cite{BJ}). Let $L\subset \mathcal G(\mathbb A_f)$ be an
open compact subgroup. Then, considering $\mathcal G(\mathbb Q)$ diagonally embedded in $\mathcal G(\mathbb A_f)$, we may
consider the intersection

$$
\Gamma_L=L\cap \mathcal G(\mathbb Q).
$$
Now, we consider $\mathcal G(\mathbb Q)$ as subgroup of $G=\mathcal G(\mathbb R)$. One easily show that the group
$\Gamma_L$ is discrete in $G$ and it has a finite covolume. The group $\Gamma_L$ is called a congruence subgroup attached to $L$.

\vskip .2in
We introduce two families of examples of congruence groups (which depend on an embedding over $\mathbb Q$ of $\mathcal G$ into
some $SL_M$).  So, in order to define a family of principal congruence subgroups
$\mathcal G$, we fix a  an embedding over $\mathbb Q$
\begin{equation}\label{int-0a}
\mathcal G \hookrightarrow SL_M
\end{equation}
with a image Zariski closed in $SL_M$. Then there exists $N\ge 1$ such that 
\begin{equation}\label{int-0b}
\text{$\mathcal G$ is a group scheme over $\bbZ[1/N]$ and the embedding  (\ref{int-0a}) is  defined over $\bbZ[1/N]$.}
\end{equation} 
We fix such $N$.

As usual, we let $\mathcal G_\bbZ=\mathcal G(\mathbb Q)\cap SL_M(\mathbb Z)$, and $\mathcal G_{\mathbb Z_p}=
\mathcal G(\mathbb Q_p)\cap SL_M(\mathbb Z_p)$ for all prime numbers $p$. We remark that 
$\mathcal G$ is a group scheme over $\mathbb Z_p$ and the embedding  (\ref{int-0a}) is  defined over $\mathbb Z_p$ when
$p$ does not divide  $N$. Then $\mathcal G_{\mathbb Z_p}=\mathcal G(\mathbb Z_p)$ but
$\mathcal G$ may not be unramified over such $p$.  In general, 
$\mathcal G_{\mathbb Z_p}$ is just an open compact subgroup of $\mathcal G(\mathbb Q_p)$.

Now, we are ready to define the principal congruence subgroups with respect to the embedding (\ref{int-0a})
\begin{equation}\label{int-0c}
\Gamma(n)=\{x=(x_{i,j})\in \mathcal G_{\mathbb Z}: \ \ x_{i,j}\equiv \delta_{i,j} \ (mod \ n)\}, \ \ n\ge 1.
\end{equation}
Obviously, $\Gamma(1)=\mathcal G_{\mathbb Z}$, and when $n | m$ we have $\Gamma(m)\subset \Gamma(n)$. One easily check that
$\Gamma(n)$ correspond to an open compact subgroup $L(n)$ defined as follows. For each prime number $p$ and $l\ge 1$,
we define the open--compact subgroup
$$
L(p^l)= \mathcal G(\mathbb Q_p)\cap \{x=(x_{ij})\in SL_M(\mathbb Z_p);  \ \ x_{i,j}- \delta_{i,j} \in p^l \mathbb Z_p\}.
$$
We decompose $n$ into primes numbers $n=p_1^{l_1}\cdots p_t^{l_t}$. Then, one easily sees that
$$
\Gamma(n)=\left(L(p_1^{l_1})\times \cdots \times L(p_t^{l_t})\times \prod_{p\not\in \{p_1, \ldots, p_t\}}
\mathcal G_{\mathbb Z_p}\right)\cap \mathcal G(\mathbb Q).
$$

We also define Hecke congruence subgroups which are of a rather different nature than principal congruence subgroups. 
\begin{equation}\label{int-0c-1}
\Gamma_1(n)=\{x=(x_{i,j})\in \mathcal G_{\mathbb Z}: \ \ x_{i,j}\equiv \delta_{ij}  \ (mod \ n) \ \text{for $i\ge  j$}\}, \ \ n\ge 1.
\end{equation}
Obviously, $\Gamma_1(1)=\mathcal G_{\mathbb Z}$, and when $n | m$ we have $\Gamma_1(m)\subset \Gamma_1(n)$.
We leave as an exercise to the reader to write down corresponding open--compact subgroups.

\vskip .2in 
We do not want to write down an exhaustive list of all applications but to give some typical results.  First, we prove the following theorem:

\begin{Thm} \label{cps-13} 
We fix an embedding $\mathcal G\hookrightarrow SL_M$ over $\mathbb Q$, and define 
Hecke congruence subgroups $\Gamma_1(n)$,
  $n\ge 1,$ using that embedding. Assume that $U$ is  a  subgroup of all upper triangular unipotent matrices in $G$ considered as $G 
\subset SL_M(\mathbb R)$. Let 
  $\chi$ be a unitary character $U\longrightarrow  \mathbb C^{\times}$ trivial on $U\cap  \Gamma_1(l)$ for some $l\ge 1$.
Let $\pi$ be an integrable discrete series of $G$  such that
 there exists  a  $K$--finite  matrix coefficient $\varphi$ of $\pi$ such that $\mathcal F_{(\chi, U)}^{loc}(\varphi)\neq 0$.
Then, there exists $n_0\ge 1$ such that for $n\ge n_0$ we have
 a realization of  $\pi$  as a $(\chi, U)$-generic cuspidal automorphic 
representation for $\Gamma_1(ln)$.
 \end{Thm}
\begin{proof} We remark that   $\Gamma_1(ln)\subset  \Gamma_1(l)$ for all $n\ge 1$. Since $U$ is a subgroup of upper--triangular unipotent matrices in 
$SL_M(\mathbb R)$, we see that 
  $U\cap \Gamma_1(ln)$ is independent of $n\ge 1$. Thus,  the same $\chi$ can be used for all $\Gamma_1(ln)$. Furthermore,  
since $\mathcal F_{(\chi, U)}^{loc}(\varphi)\neq 0$ is integrable on $U\cap \Gamma_1(ln)\backslash G$,
  there exists a compact set $C\subset G$ such that 
$$
\int_{\left( U\cap \Gamma_1(ln)\right)\backslash \left( U\cap \Gamma_1(ln)\right)\cdot C}
\left|\mathcal F_{(\chi, U)}^{loc}(\varphi)(g)\right|dg > \frac{1}{2}
\int_{ U\cap \Gamma_1(ln)\backslash G} \left|\mathcal F_{(\chi, U)}^{loc}(\varphi)(g)\right|dg,
$$
for all $n\ge 1$. This is (2) in Lemma \ref{cps-12}.

For $x=(x_{ij})\in \Gamma_1(ln)$, we have either $x_{ij}=\delta_{ij}$ or $|x_{ij}-\delta_{ij}|\ge ln$ for all  $i\ge j$. But matrix entries of 
$C\cdot C^{-1}$ are bounded. 
Thus, we see that there must exist $n_0\ge 1$ such that for $n\ge n_0$ and $x\in \Gamma_1(ln)\cap C\cdot C^{-1}$ we have
$x_{ij}=\delta_{ij}$ for all $i\ge j$. Hence, $x$ is an upper--triangular unipotent matrix. Thus, by definition of $U$, $x\in U\cap \Gamma_1(ln)$.
This is (1) in Lemma \ref{cps-12}. 

Now, Lemma \ref{cps-12} implies that 
$$
  W^{\Gamma_1(ln)}_{(\chi, U)}\left(\mathcal F_{(\chi, U)}^{loc}(\varphi)\right)\neq 0, \ \ n\ge n_0.
$$
Thus, Proposition \ref{cps-10} implies that there exists  a realization of  $\pi$  as a $(\chi, U)$-generic cuspidal automorphic 
representation for $\Gamma_1(ln)$ for all $n\ge n_0$.
\end{proof}

\vskip .2in
Now, we give a global application (see \cite{MuicMathAnn}).
We briefly recall the notion of a cuspidal automorphic form \cite{BJ}, or (\cite{MuicMathAnn},
Section 1). A cuspidal automorphic form is a function $f: \mathcal G(\mathbb A)\longrightarrow \mathbb C$ which satisfies
the following conditions:
\begin{itemize}
\item[(C-1)] $f(\gamma x)=f(x)$, for all $\gamma\in \mathcal G(\mathbb Q)$, $x\in \mathcal G(\mathbb A)$;
\item[(C-2)] If we write $\mathcal G(\mathbb A)= G\times \mathcal G(\mathbb A_f)$, then $f$ is
  $\mathcal Z(\mathfrak g_{\mathbb C})$--finite, and
  $K$--finite  on the right in the first variable, and there exists an open--compact subgroup $L$ such that
  $f$ is right--invariant under $L$ in the second variable;
\item[(C-3)] $f\in L^2( \mathcal G(\mathbb Q)\backslash  \mathcal G(\mathbb A)$ (which implies that $f$ is of moderate growth
  (as in Lemma \ref{paf-1}));
\item[(C-4)]  The cuspidality condition holds: $\int_{\mathcal U(\mathbb Q)\backslash \mathcal U(\mathbb A)}f(ux)du=0$, for all
  $x\in \mathcal G(\mathbb A)$, where $\mathcal U$ is the unipotent radical of a proper $\mathbb Q$--parabolic subgroup
  of $\mathcal G$.
  \end{itemize}

\vskip .2in

The space $\mathcal A_{cusp}(\mathcal G(\mathbb Q) \backslash \mathcal G(\mathbb A))$ is a 
$(\mathfrak g, K)\times \mathcal G(\mathbb A)$--module. If
$f\in \mathcal A_{cusp}(\mathcal G(\mathbb Q) \backslash \mathcal G(\mathbb A))$ is a non--zero, then 
$(\mathfrak g, K)\times \mathcal G(\mathbb A)$--module generated by $f$ is a direct sum of finitely many irreducible modules
(they consists of cuspidal automorphic forms) (see \cite{BJ} for details).

\vskip .2in
Now, we assume that $\mathcal G$ is quasi--split
  over $\mathbb Q$. Let $\mathcal N$ be the unipotent radical of a Borel subgroup defined over $\mathbb Q$.  We assume that
  $G$ poses representations in discrete series. Let $L$ be the $L$--packet of discrete series for  $G$ such that some large
  representation in that packet is integrable. Let $\eta: \mathcal N(\mathbb Q) \backslash \mathcal N(\mathbb A)\longrightarrow
  \mathbb C^\times$ be a unitary and non--degenerate character (i.e., it is not trivial by restriction to
  a root subgroup of simple root). For any open--compact subgroup $L\subset \mathcal G(\mathbb A_f)$, we have
  $\mathcal N(\mathbb A_f)=\mathcal N(\mathbb Q) \cdot \left(L\cap \mathcal N(\mathbb Q)\right)$ (where again
  $\mathcal N(\mathbb Q)$ is
  diagonally embedded in $\mathcal N(\mathbb A_f)$) by strong approximation.  Assume that $\eta$ is right--invariant under
  $\mathcal N(\mathbb A_f)\cap L$ (at least one such $L$ exists), then by writing $\mathcal N(\mathbb A)= \mathcal N(\mathbb R)
  \times \mathcal N(\mathbb A_f)$,  we see that a character $\eta_\infty$ defined by
  $\eta_\infty(n)=\eta(n, 1)$ is trivial on $\Gamma_L\cap  \mathcal N(\mathbb R)$. The obtained character is generic in sense of
  definition introduced in the statement of Corollary \ref{gwm-6}.

  \vskip .2in 
  A cuspidal automorphic form $f$ is $\eta$--generic (or $(\eta, \mathcal N)$--generic) if
  $(\eta, \mathcal N)$--Fourier coefficient defined by 
  $$
  \int_{\mathcal N(\mathbb Q)\backslash \mathcal N(\mathbb A)}f(nx) \overline{\eta(n)} dn, \ \ x\in \mathcal G(\mathbb A),
  $$
  is not identically equal to zero. A submodule $W$ of
  $\mathcal A_{cusp}(\mathcal G(\mathbb Q) \backslash \mathcal G(\mathbb A))$ is $\eta$--generic if there exists a cuspidal
  automorphic form $f\in W$ such that it is $\eta$--generic.

  \vskip .2in

  Let  $f$ be a cuspidal automorphic form. Let us select an open--compact
  subgroup $L\subset  \mathcal G(\mathbb A_f)$ such that $f$ is right--invariant under
  $L$ in the second variable (see (C-2)), and such that $\eta$ is right--invariant under  $\mathcal N(\mathbb A_f)\cap L$
  in the second variable also. Then, by (\cite{MuicComp}, Lemma 3.3), we have

  \begin{equation}\label{ime}
  \int_{\mathcal N(\mathbb Q)\backslash \mathcal N(\mathbb A)}f(nx) \overline{\eta(n)} dn=
  vol_{\mathcal N(\mathbb A_f)}\left(\mathcal N(\mathbb A_f)\cap L\right) \int_{\mathcal N(\mathbb R)\cap \Gamma_L\backslash
   \mathcal N(\mathbb R)}f(nx)\overline{\eta_\infty(n)} dn, 
  \end{equation}
  for all  $x\in G$.
  
  \vskip .2in
  Now, we are ready to state and prove the following global result concerning the 
  realization of generic integrable discrete series in the space of generic cuspidal automorphic forms.

  \begin{Thm}\label{cps-16} Let $\mathcal G$ be a semisimple algebraic group defined over $\mathbb Q$.
  Assume that  $G=\mathcal G(\mathbb R)$ is connected. In addition, we assume that $\mathcal G$ is quasi--split
  over $\mathbb Q$. Let $\mathcal N$ be the unipotent radical of Borel subgroup defined over $\mathbb Q$.  We assume that
  $G$ poses representations in discrete series. Let $L$ be the $L$--packet of discrete series for  $G$
  such that there is a
  large
  representation in that packet which is integrable (then all are integrable by Proposition \ref{integrable-4}).
  Let $\eta: \mathcal N(\mathbb Q) \backslash \mathcal N(\mathbb A)\longrightarrow
  \mathbb C^\times$ be a unitary generic character. By the change of splitting we can select an
  $(\eta_\infty, \mathcal N(\mathbb R))$--generic discrete series $\pi$ in the $L$--packet $L$. Then, there exists
  a cuspidal automorphic module $W$ for $\mathcal G(\mathbb A)$ which is $\eta$--generic and its Archimedean component is
  infinitesimally equivalent to $\pi$ (i.e., considered as a $(\mathfrak g, K)$--module only it is a copy of (infinitely many)
  $(\mathfrak g, K)$--modules infinitesimally equivalent to $\pi$).
    \end{Thm}
  \begin{proof} Let $\varphi$ be a non--zero $K$--finite matrix coefficient of $\pi$. Let $L\subset \mathcal G(\mathbb A_f)$ 
be an open--compact subgroup such that $\eta$ is right--invariant under
    $\mathcal N(\mathbb A_f)\cap L$. Let $1_L$ be the characteristic function of $L$ in $\mathcal G(\mathbb A_f)$. 
    By (\cite{MuicMathAnn}, Theorem 3-10), the Poincar\' e series
    $$
    P\left( \varphi \otimes 1_L \right)(x)=\sum_{\gamma\in \mathcal G(\mathbb Q)} \left(\varphi \otimes 1_L\right) (\gamma\cdot x),
    \ \ x\in \mathcal G(\mathbb A)
    $$
    converges absolutely  and uniformly on compact sets  to a function in
    $\mathcal A_{cusp}(\mathcal G(\mathbb Q) \backslash \mathcal G(\mathbb A))$. 

    Obviously $P\left(\varphi \otimes 1_L\right)$ is right--invariant under $L$. Hence, (\ref{ime}) is applicable and
    we obtain

    \begin{equation}\label{ime-1}
      \begin{aligned}
&  \int_{\mathcal N(\mathbb Q)\backslash \mathcal N(\mathbb A)}P\left(\varphi \otimes 1_L\right)(nx) \overline{\eta(n)} dn =\\
& = vol_{\mathcal N(\mathbb A_f)}\left(\mathcal N(\mathbb A_f)\cap L\right) \int_{\mathcal N(\mathbb R)\cap \Gamma_L\backslash
    \mathcal N(\mathbb R)}P\left(\varphi \otimes 1_L\right)(nx, 1) \overline{\eta_\infty(n)} dn\\
  &= vol_{\mathcal N(\mathbb A_f)}\left(\mathcal N(\mathbb A_f)\cap L\right) \int_{\mathcal N(\mathbb R)\cap \Gamma_L\backslash
    \mathcal N(\mathbb R)}P_{\Gamma_L}(\varphi)(nx) \ \overline{\eta_\infty(n)} dn, \ \ x\in G,
  \end{aligned}
  \end{equation}
    since for $x\in G$ using  $\mathcal G(\mathbb A)= G\times \mathcal G(\mathbb A_f)$ we have
    $$
    P\left(\varphi \otimes 1_L\right)(x, 1)= \sum_{\gamma\in \mathcal G(\mathbb Q)} \left(\varphi \otimes 1_L\right) (\gamma x, \gamma)
    = \sum_{\gamma\in \Gamma_L} \varphi (\gamma x )\overset{def}=P_{\Gamma_L}(\varphi)(x).
    $$

    We will show that  we can choose $\varphi$ such that after  shrinking  $L$  we have   (see Definition \ref{cps-1000})

       \begin{equation}\label{cps-17}
    \mathcal F_{(\eta_\infty, \mathcal N(\mathbb R))}(P_{\Gamma_L}(\varphi))(1)=   \int_{\mathcal N(\mathbb R)\cap \Gamma_L\backslash
      \mathcal N(\mathbb R)}P_{\Gamma_L}(\varphi)(n) \ \overline{\eta_\infty(n)} dn\neq 0.
    \end{equation}
    Then, in view of (\ref{ime-1}) we can find required $W$ among finitely many irreducible submodules that
    make in a direct sum the module generated by $P\left(\varphi \otimes 1_L\right)$.

    So, let us prove that there exists a $K$--finite matrix coefficient $\varphi$ such that (\ref{cps-17}) hold. We may  assume that
    $L$ is factorizable. Then, there exists primes $p_1, \ldots, p_l$, and open--compact subgroups $L_i\subset \mathcal
    G(\mathbb Q_{p_i})$, for $i=1, \ldots, l$, such that
    $$
    L=L_1\times \cdots \times L_l \times \prod_{p\not\in \{p_1, \ldots, p_t\}} \mathcal G_{\mathbb Z_p}.
    $$

    We  fix prime $p_1$ and shrink $L_1$ without altering other groups in the decomposition of $L$. We fix a sufficiently large 
natural number $t$ such that
    $$
    L_1'\overset{def}{=} \mathcal G(\mathbb Q_{p_1})\cap \{x=(x_{ij})\in SL_M(\mathbb Z_{p_1});  \ \ x_{i,j}- \delta_{i,j} \in
    p^t_1 \mathbb Z_{p_1}\}\subset L_1.
$$
    
Next, for each $u\ge t$ we define open--compact subgroups (we ''shrink'' the lower part of $L_1'$)
    \begin{align*}
    L_1'(u) &\overset{def}{=} \mathcal G(\mathbb Q_{p_1})\cap \\
&\cap \{x=(x_{ij})\in SL_M(\mathbb Z_{p_1});  \ \ x_{i,j} \in
    p^t_1 \mathbb Z_{p_1} \ \  \text{for $j>i$}, \ \   x_{i,j}- \delta_{ij} \in
    p^u_1 \mathbb Z_{p_1} \ \ \text{for $i\ge j$} \}.
    \end{align*}
We consider the following family of open--compact subgroups:
     $$
    L(u)=L'_1(u)\times L_2 \times \cdots \times L_l \times \prod_{p\not\in \{p_1, \ldots, p_t\}} \mathcal G_{\mathbb Z_p}, \ \  u\ge t.
    $$
    Next, we may assume that  embedding (\ref{int-0a}) such that $\mathcal N$ embeddes into unipotent upper triangular matrices.
Then, $\mathcal N(\mathbb R)$ consist of all unipotent upper triangular matrices in $G$. 
Also,  as in the proof  of Theorem \ref{cps-13}, we have that $\mathcal N(\mathbb R)\cap \Gamma_{L(u)}$ is independent of
    $u\ge t$. This will be applied in the following way.    By Corollary \ref{gwm-6}, there exists a  $K$--finite matrix
    coefficient $\varphi_1$ of $\pi$ which satisfies
   $$
\mathcal F_{(\eta_\infty, \ U)}^{loc}(\varphi_1)\neq 0
$$
Now, as in the proof of Theorem \ref{cps-13}, there exists $u_0\ge t$ such that the series (see Lemma \ref{cps-7}) 
$$
W^{\Gamma_{L(u)}}_{(\eta_\infty, \ \mathcal N(\mathbb R))}\left(\mathcal F_{(\eta_\infty, \ U)}^{loc}(\varphi_1)\right)(x)=
\sum_{\gamma \in \mathcal N(\mathbb R) \cap \Gamma_{L(u)} \backslash \Gamma_{L(u)}}
\mathcal F_{(\eta_\infty, \ U)}^{loc}(\varphi_1)(\gamma x), \ \ x\in G,
$$
is non--zero for $u\ge u_0$. Now, we let $u=u_0$, and as in the proof of Proposition  \ref{cps-10}, we construct required
$\varphi$ (i.e., the one satisfying (\ref{cps-17})) exists. This completes the proof of the theorem. 
    \end{proof}

  \vskip .2in
  Final application concerns the group $Sp_{2n}(\mathbb R)$ and Proposition \ref{cps-100}. We choose $\mathbb Q$--embedding
  $Sp_{2n}\hookrightarrow SL_{2n}$ (see (\ref{int-0a})) such that the Borel subgroup in $Sp_{2n}$ is the set of upper triangular
  matrices in $Sp_{2n}$ this requires to conjugate by an element in $GL_{2n}(\mathbb Q)$ the standard realization
  $$
 \left\{g\in GL_{2n}; \ \ g\begin{pmatrix} 0 & I_n \\ -I_n & 0\end{pmatrix} g^t=  \begin{pmatrix} 0 & I_n \\ -I_n & 0\end{pmatrix}\right\}.
  $$
    By that conjugation we transfer back to the standard realization the principal congruence subgroups $\Gamma(n)$, defined
    above,  and called them by the same name. Let $P=MN$ be the standard Siegel parabolic subgroup of
    $Sp_{2n}(\mathbb R)$, where $M$ consists of matrices $\begin{pmatrix} x & 0 \\ 0 & (x^t)^{-1}\end{pmatrix}$,
    $x\in GL(n, \mathbb R)$, and $N$ is the unipotent radical consists of all matrices $\begin{pmatrix} I_n & x \\ 0 &
      I_n\end{pmatrix}$, where $x$ is a symmetric matrix of size $n\times n$  with real entries. Obviously, every unitary character of $N$ is of the form
      $$
      x\longmapsto \exp{\left(2\pi \sqrt{-1} \trace{(xy)}\right)}
        $$
      for unique  symmetric matrix $y$ of size $n\times n$  with real entries. We denote this character by $\chi_y$. The character  $\chi_y$
      is non--degenerate if $\det{y}\neq 0$ \cite{Li}.

Let $m\ge 1$. Then, if we identify $N$ with symmetric matrices  of size $n\times n$  with real entries, then $N\cap \Gamma(m)$ gets identified 
with symmetric matrices  of size $n\times n$  with  entries in $m\mathbb Z$.  This implies that every character of $N\cap \Gamma(m)\backslash N$
is of the form $\chi_{\frac{1}{m}y}$ for unique symmetric matrix $y$ of size $n\times n$  with integral entries.

\vskip .2in 
\begin{Thm}\label{cps-18} Let $G=Sp_{2n}(\mathbb R)$. Assume that $\pi$ is an integrable discrete series for
$G$. Let $\varphi$ be any non--zero $K$--finite matrix coefficient of $\pi$.
We fix an infinite  sequence of natural numbers $n_k$, $k\ge 1$, such that $n_k| n_{k+1}$, for all $k\ge 1$. 
Then, there exists $k_0$ such that for each $k\ge k_0$, there exists a symmetric invertible matrix $y_k$  of size $n\times n$  
with integral entries such that  
$$
\mathcal F^{loc}_{(\chi_{y_k},\  N)}  \left(\varphi\right)\neq 0.
$$
In particular,  $\pi$ is infinitesimally equivalent to a subrepresentation of $I^1(G, N, \chi_{\frac{1}{m}y_k})$.
\end{Thm}
\begin{proof}  We consider the Poincar\' e series
$$
 P_k(\varphi)(x)=\sum_{\gamma \in \Gamma(n_k)} \varphi(\gamma x), \ \ x\in G,
 $$
for all $k\ge 1$. Since  $n_1| n_2 | n_3 \cdots$, we obtain that 
$$
\Gamma(n_1)\supset \Gamma(n_2) \supset \cdots,
$$
and 
$$
\cap_{k\ge 1} \Gamma(n_k)=\{1\}.
$$
Hence, Lemma \ref{cps-2} (iii) implies that there exists $k_0$  such that for each $k\ge k_0$ we have $P_k(\varphi)\neq 0$. Now,
 we apply Proposition \ref{cps-100} to  see that $\pi$ is infinitesimally equivalent to a subrepresentation of $I^1(G, N, \chi_k)$, 
for some character $\chi_k$ of $N\cap \Gamma(n_k)\backslash N$. 

In fact, the set of all such characters is exactly the set of characters of  $N\cap \Gamma(n_k)\backslash N$  for which the
 corresponding Fourier coefficients of  $P_k(\varphi)$ are non--zero. Some of them must be  non--degenerate 
since the same holds for the following cuspidal automorphic form in global adelic settings \cite{Li}:
 $$
    P\left( \varphi \otimes 1_{L(n_k)} \right)(x)=\sum_{\gamma\in  Sp_{2n}(\mathbb Q)} \left(\varphi \otimes 1_{L(n_k)}\right) (\gamma\cdot x),
    \ \ x\in  Sp_{2n}(\mathbb A)
    $$
(see (\cite{MuicMathAnn}, Theorem 3-10) for the treatment of such forms). Here $L(n_k)$ is an open compact subgroup $Sp_{2n}(\mathbb A_f)$
such that 
$$
\Gamma(n_k)=\Gamma_{L(n_k)}.
$$
Finally, we transfer this result to $P_k(\varphi)$ using analogue of (\ref{ime-1}) (see the proof of Theorem \ref{cps-16}).
\end{proof}

\section{Appendix: Some Remarks on Integrable Discrete Series}\label{integrable}

This section is based on the correspondence with Mili\v ci\' c \cite{milicic1}. The main purpose of this section to explain that 
many discrete series are integrable and are
$(\chi, N)$--generic for some generic unitary character $\chi$ of the unipotent radical $N$ of
a minimal parabolic subgroup (see Definition \ref{gwm-1001}, and the definition included in the statement of Corollary
\ref{gwm-6}).

Assume that $G$ admits discrete series. By the well--known result of Harish--Chandra this is equivalent to
$rank(G)=rank(K)$. We recall their classification following the introduction of \cite{milicic} with almost the same notation.
We remind the reader that $G$ is a connected semisimple  Lie group with finite center.

Let $H$ be the compact Cartan subgroup of $G$. We write $\mathfrak h=Lie(H)$. In Section \ref{gwm}, we have introduce
$\mathfrak k=Lie(K)$. We denote by $\Phi$ the set of roots of $\mathfrak h_{\mathbb C}$ inside $\mathfrak g_{\mathbb C}$.
The root $\alpha\in \Phi$ is compact if the corresponding (one--dimensional) root subspace belong to $\mathfrak k_{\mathbb C}$.
Otherwise, it is non--compact.

Let $W$ be the Weyl group of  $\mathfrak h_{\mathbb C}$ in  $\mathfrak g_{\mathbb C}$ 
and $W_K$  its subgroup generated by the reflections with respect to the compact
roots. The Killing form of  $\mathfrak g_{\mathbb C}$  induces an inner product $(\ | \ )$ on $\sqrt{-1} \mathfrak h^*$,
the space of all linear forms on  $\mathfrak h_{\mathbb C}$  which assume imaginary values of  $\mathfrak h$. An element
$\lambda$ of $\sqrt{-1} \mathfrak h^*$ is singular if it is orthogonal to at least one root in, and nonsingular
otherwise. The differentials of the characters of $H$ form a lattice $\Lambda$ in $\sqrt{-1} \mathfrak h^*$.
Let $\rho$  be the half-sum of positive roots in $\Phi$, with respect to some ordering on
$\sqrt{-1} \mathfrak h^*$. Then $\Lambda+\rho$ does not depend on the choice of this ordering.

Harish-Chandra has shown that to each nonsingular $\lambda\in \Lambda+\rho$  we can attach
a discrete series representation $\pi_\lambda$, of infinitesimal character $\mu_\lambda$ (usual parameterization)
corresponding to $\lambda$, with $\pi_\lambda\simeq  \pi_\nu$ if and only if $\lambda\in W_K\nu$.
In this way, all discrete series are obtained up to equivalence.

By Langlands, discrete series are divided into $L$--packets according to their infinitesimal characters:
to infinitesimal character $\mu$, we define the $L$--packet by $L_\mu=\{\pi_\lambda; \ \ \mu=\mu_\mu\}$.
By above, if we fix $\pi_\lambda\in L_\mu$, then $L_\mu=\{\pi_{w(\lambda)}; \ \ w\in W\}$ has  $\#(W/W_K)$--elements.

Following \cite{milicic}, we define
$$
\kappa(\alpha)=\frac{1}{4} \sum_{\beta\in \Phi} |(\alpha| \beta)|, \ \alpha\in \Phi.
$$
By the paragraph after the statement of the theorem in the Introduction of  \cite{milicic},
the discrete series $\pi_\lambda$ is integrable if and only if
\begin{equation}\label{integrable-1} 
|(\lambda| \alpha)|> \kappa(\alpha), \ \ \text{for all non--compact roots $\alpha\in \Phi$}.
\end{equation}
We remark that $W_K$ transforms non--compact roots onto non--compact roots. Therefore,
if $\lambda$ satisfies (\ref{integrable-1}), then $w(\lambda)$ satisfies the same. In particular, the condition
(\ref{integrable-1}) is independent of the way we write $\pi$ in the form $\pi_\lambda$.

Roughly speaking, in the language of \cite{milicic}, $\pi_\lambda$ is integrable if and only if
$\lambda$ is far enough from non--compact walls.

\begin{Def}\label{integrable-2} Let $\pi$ be an integrable discrete series  representation, let us write $\pi=\pi_\lambda$.
  We let
  $$
  \epsilon_\pi=\min_{\substack{\alpha\in \Phi\\ \text{$\alpha$ is non--compact}}}
    \frac{|(\lambda| \alpha)|}{\kappa(\alpha)}.
    $$
   \end{Def}
Again, this is independent of the way we write $\pi$ as $\pi_\lambda$

\vskip .2in

Since we are working with integrable representations, we include the following lemma which
shows to some extent how big is the set of integrable representations in terms of $L$--packets.
The lemma is an easy consequence above results of Mili\v ci\' c. We leave the proof as an exercise to the reader.

\begin{Lem}\label{integrable-3} Let $\lambda\in \Lambda+\rho$ be such that
  $|(\lambda| \alpha)|> \kappa(\alpha)$, for all roots $\alpha\in \Phi$. Then, all representations in the $L$--packet
  $L_{\mu_\lambda}$ are integrable.
 \end{Lem}

We remark that the following is an easy way to see that $\lambda$'s as in Lemma \ref{integrable-3} exists. Let $C$ be the
Weyl chamber used to determine positive roots of $\Phi$ used to compute $\Lambda$. We select any $\lambda_0\in
\left(\Lambda+\rho\right)\cap C$. Then, for sufficiently large odd integer $l$, we have that $\lambda=l\lambda_0$ satisfies
the requirements of the lemma.

\vskip .2in 
 Let $(\pi, \mathcal H)$ be an irreducible unitary representation on the Hilbert space
$\mathcal H$. We say that $(\pi, \mathcal H)$ is large if the annhilator of $\mathcal H_K$ in $\mathcal U(\mathfrak g_{\mathbb C})$
 is of maximal Gelfand--Kirillov dimension (i.e., a minimal primitive ideal) \cite{vogan}. This implies that
 $G$ is quasi--split (\cite{vogan}, Corollary 5.8). From now on until the end of this section, 
we assume that $G$ is quasi--split.

 The reformulation of largeness in the analytic language is due to Kostant
\cite{kostant}. The representation  $(\pi, \mathcal H)$ is large if and only if it is
$(\chi, N)$--generic for some generic unitary character $\chi$ of the unipotent radical $N$ of
minimal parabolic subgroup (see Definition \ref{gwm-1001}, and the definition included in the statement of Corollary
\ref{gwm-6}).

\vskip .2in
By (\cite{vogan}, Theorem 6.2) (or \cite{milicic1} for more conceptual proof). The discrete series representation $\pi$ 
is large if and only if the following holds. Let us write $\pi=\pi_\lambda$. Since $\lambda$ is non--singular, it determines the set of
positive roots $\Phi^+_\lambda$ by $\alpha\in \Phi^+_\lambda$ if and only if $(\lambda | \alpha)>0$. Now, $\pi_\lambda$ is large
if and only if the basis of $\Phi^+_\lambda$ consists of non--compact roots. Clearly, this claim does not depend on
choice of $\lambda$ used to write  $\pi=\pi_\lambda$. Every $L$--packet of discrete series contains large representations
\cite{vogan} (or \cite{milicic1} which shows that this statement reduces to certain facts from the structure theory of
Lie algebras). By above mentioned test for largeness we can easily describe  all large representations in the $L$--packet.
The details are left to the reader as an exercise. In \cite{milicic1}, Mili\v ci\' c has observed the following:

\begin{Prop}\label{integrable-4} A large representation $\pi$ is integrable if and only if its $L$--packet is of the form
   considered in  Lemma \ref{integrable-4}. If this is so, all  representations in the $L$--packet are integrable.
\end{Prop}
\begin{proof} It remains to see that 'only if' part. 
  Let $\pi_\lambda$ be large and integrable.  Then, by Mili\v ci\' c's criterion,
  $\pi_\lambda$ is integrable if and only if $|(\lambda| \alpha)|> \kappa(\alpha)$, for all non--compact roots $\alpha\in \Phi$.
  On the other hand, $\pi_\lambda$ is large, so by the criterion, this is equivalent to the fact that the basis $\{\alpha_1,
  \ldots, \alpha_l\}$ of $\Phi^+_\lambda$ consists of non--compact roots. So, in particular,
  $(\lambda| \alpha_i)> \kappa(\alpha_i)$, for $i=1,\ldots, l$. Then, if $\alpha=\sum_{i=1}^l k_i \alpha_i$ is any root in
  $\Phi^+_\lambda$, then
  $$
  (\lambda| \alpha)= \sum_{i=1}^l k_i  (\lambda| \alpha_i)>  \sum_{i=1}^l k_i  \kappa(\alpha_i)=
  \frac{1}{4} \sum_{\beta\in \Phi}\left( \sum_{i=1}^l k_i |(\alpha_i| \beta)|\right)\ge
   \frac{1}{4} \sum_{\beta\in \Phi}\left|\left( \sum_{i=1}^l k_i \alpha_i| \beta\right) \right|=\kappa(\alpha).$$
  So, we are in the assumptions of Lemma \ref{integrable-3}. 
\end{proof}

In closing this section,  we mention that  Proposition \ref{integrable-4}  shows that the great many representations
in discrete series are large and integrable. In view of results of Kostant \cite{kostant} mentioned above, this shows that
the results that we present in this paper are for a non--trivial number of discrete series.


\begin{thebibliography}{999999}

  

\bibitem{Borel1963}
{\sc A.~Borel,} {\em Some finiteness properties of adele groups over number fields,} Publ. Math.
Inst. Hautes \' Etudes Sci. {\bf 16} (1963), 5-30.

\bibitem{Borel1966}
{\sc A.~Borel,} {\em Introduction to automorphic forms 
(Proc. Sympos. Pure Math., Oregon State Univ., Corvallis, Ore., 1977),
Part 1,} Proc. Sympos. Pure Math. {\bf  IX},  Amer. Math. Soc., Providence, R.I.
(1966), 199--210.

\bibitem{Borel-SL(2)-book}{\sc A.~Borel,} {\em 
Automorphic forms on $SL_2 (\mathbb R)$ (Cambridge Tracts in Mathematics),}
Cambridge University Press (1997).
 

\bibitem{bb}{\sc W.~L.~Baily,  A.~Borel,} {\em Compactification of
    arithmetic quotients of bounded symmetric domains,} Annals of
    Math., Vol. {\bf 84} (1966), 442--528.


\bibitem{BJ}
{\sc A.~Borel, H.~Jacquet,} {\em Automorphic forms and automorphic representations
(Proc. Sympos. Pure Math., Oregon State Univ., Corvallis, Ore., 1977),
Part 1,} Proc. Sympos. Pure Math. {\bf  XXXIII},  Amer. Math. Soc., Providence, R.I.
(1979), 189--202.

\bibitem{casselman}{W.~Casselman,}
  {\em  Canonical extensions of Harish-Chandra modules to
    representations of G. Canad. J. Math. 41 (1989), no. 3,
    385–-438.} 

\bibitem{clozel} {\sc L.~Clozel,}  {\em On limit multiplicities of discrete series representations in
  spaces of automorphic forms.} Invent. Math. {\bf 83} (1986), no. 2, 265–284.


\bibitem{finis}{\sc T.~Finis,}{\em The limit multiplicity problem for congruence subgroups of arithmetic lattices and the
  trace formula,} RIMS Kokyuroku 1871, December 2013.

\bibitem{grs} {\sc D.~Ginzburg, S.~Rallis, D.~Soudry,} {\em
  On Fourier coefficients of automorphic forms of symplectic groups.} Manuscripta Math. {\bf 111} (2003), no. 1, 1–-16. 
  

\bibitem{gourevitch} {\sc D.~Gourevitch, S.~Sahi,} {\em Degenerate Whittaker functionals for real reductive groups,}
  http://www.wisdom.weizmann.ac.il/~dimagur/DimaTalkDegWhit.pdf.

\bibitem{degeorgy} {\sc D.~de~George, N.~Wallach}{\em Limit formulas for multiplicities in $L^2(\Gamma \backslash )$.}
  Ann. of Math. {\bf (2) 107} (1978), no. 1, 133–-150.
  

\bibitem{hc}{\sc Harish--Chandra,} {\em Discrete series for semisimple
    Lie groups II,} Acta Math. {\bf 116} (1966), 1-111.

\bibitem{kls}{\sc C.~Khare, M.~Larsen, G.~Savin,} {\em Functoriality and the inverse Galois problem,}
Compositio Math. {\bf 144} (2008), 541–-564.

\bibitem{kostant}{\sc B.~Kostant,} {\em On Whittaker vectors and representation theory,} Invent. Math. {\bf 48} (1978),
  no. 2, 101–-184.

  
\bibitem{Li} {\sc J.-S.~Li,} {\em Nonexistence of singular cusp forms,} Compositio Math. {\bf 83} (1992), no. 1, 43--51. 
    


\bibitem{milicic}
{\sc D.~Mili\v ci\' c,} {\em Asymptotic behavior of matrix coefficients of
  the discrete series,} Duke Math. J {\bf 44} (1977), 59--88.

\bibitem{milicic1}
{\sc D.~Mili\v ci\' c,} {\em A letter to the author,} April 2015.

\bibitem{MuicLie} {\sc G.~Mui\' c,} {\em On the decomposition of $L^2(\Gamma\backslash G)$ in the cocompact case.}
J. Lie Theory  {\bf 18}  (2008),  no. 4, 937--949.

\bibitem{MuicMathAnn} {\sc G.~Mui\' c,} {\em
On a Construction of Certain Classes of Cuspidal Automorphic Forms via Poincare Series,} Math. Ann. {\bf 343}, No.1 (2009), 
207--227.


\bibitem{MuicJNT} {\sc G.~Mui\' c,} {\em On the Cuspidal Modular Forms for the Fuchsian Groups of the First Kind,}
  Journal of Number Theory Volume {\bf 130, Issue 7}, July 2010, Pages 1488-1511

\bibitem{MuicComp} {\sc G.~Mui\' c,} {\em Spectral Decomposition of Compactly Supported Poincar\' e
Series and Existence of Cusp Forms,}  Compositio Math. {\bf 146, No. 1} (2010) 1-20.

\bibitem{Muic2} {\sc G.~Mui\' c,} {\em
On the Cusp Forms  for the Congruence Subgroups of $SL_2(\mathbb R)$,}
Ramanujan J. {\bf 21}  No. 2 (2010), 223-239.

\bibitem{MuicIJNT} {\sc G.~Mui\' c,} {\em
On the Non--Vanishing of Certain Modular Forms , } International J. of Number Theory {\bf Vol. 7, } Issue 2 (2011) pp. 351-370.


\bibitem{rospeh} {\sc J.~Rohlfs, B.~Speh,} {\em
  On limit multiplicities of representations with cohomology in the cuspidal spectrum.} Duke Math. J. {\bf 55}
  (1987), no. 1, 199–-211.

  
\bibitem{savin}
  {\sc G.~Savin,} {\em Limit multiplicities of cusp forms.} Invent. Math. {\bf 95} (1989), no. 1, 149–159.
  


\bibitem{sh}
{\sc F.~Shahidi,} {\em A proof of Langlands conjecture on Plancherel measures;
complementary series for $p$-adic groups,}
Annals of Math. {\bf 132} (1990), 273-330.

\bibitem{soudry} {\sc D.~Soudry,} {\em
  Automorphic descent: an outgrowth from Piatetski-Shapiro's vision. Automorphic forms and related geometry:
  assessing the legacy of I. I. Piatetski-Shapiro,} 407–432, Contemp. Math., {\bf 614}, Amer. Math. Soc., Providence, RI, 2014. 

\bibitem{vogan}{sc D.~Vogan,} {\em Gelʹfand-Kirillov dimension for Harish-Chandra modules,}
  Invent. Math. {\bf 48} (1978), no. 1, 75–-98. 

\bibitem{W1} {\sc N.~R.~Wallach,} {\em Real reductive groups I,}  
 Academic Press, Boston, 1988. 

\bibitem{W2} {\sc N.~R.~Wallach,} {\em Real reductive groups II,}
 Academic Press, Boston, 1992.


\bibitem{W3} {\sc N.~R.~Wallach,} {\em A lecture delievered at W. Schmid's birthday conference (2013),}
  http://www.math.harvard.edu/conferences/schmid\_2013/docs/wallach.pdf.

  \bibitem{Tits} {\sc J.~Tits,} {\em
Reductive groups over local fields. Automorphic forms,
representations and $L$-functions,} Proc. Sympos. Pure Math. 
{\bf XXXIII}, Amer. Math. Soc., Providence, R.I (1979), 29--69.

\end{thebibliography}
\end{document}